\documentclass{amsart}
\usepackage{amsfonts}
\usepackage{amsthm}
\usepackage{amstext}
\usepackage{amsmath}
\usepackage{amscd}
\usepackage{amssymb}
\usepackage[mathscr]{eucal}
\usepackage{mathrsfs}
\usepackage{epsf}
\usepackage{array}
\usepackage{url}
\usepackage{enumerate}
\usepackage{enumitem}
\usepackage{graphicx}
\usepackage{accents}
\usepackage{hyperref}

\newcommand{\NNN}{\mathbb{N}}
\newcommand{\RRR}{\mathbb{R}}
\newcommand{\card}{\#}    

\newcommand{\AAa}{A}      
\newcommand{\LLl}{\mathcal{L}}      
\newcommand{\eps}{\varepsilon}
\newcommand{\ubar}[1]{\text{\b{$#1$}}}

\newcommand{\deter}{\textup{DET}}
\newcommand{\rr}{\textup{RR}}
\newcommand{\lavg}{\textup{LAVG}}
\newcommand{\ent}{\textup{ENT}}
\newcommand{\corsum}{\textup{C}}
\newcommand{\Orb}{\textup{Orb}}
\newcommand{\supp}{\textup{supp}}

\newcommand{\kkk}{\mathcal{K}}

\newcommand{\dens}{{d}}
\newcommand{\linedens}{\textup{d}}
\newcommand{\lineDens}{\tilde{\textup{d}}}
\newcommand{\recog}{\varrho}
\newcommand{\numberofdiag}{\textup{L}}

\newcommand{\lang}{\mathscr{L}}    
\newcommand{\recplot}{\mathcal{R}}

\newcommand{\word}[3]{#1_{[ #2 , #3)}}

\newcommand{\abs}[1]{|#1|}

 \makeatletter
 \def\@seccntformat#1{\csname the#1\endcsname.\quad}
 \makeatother

\theoremstyle{plain}
\newtheorem{theorem}{Theorem}[section]
\newtheorem{proposition}[theorem]{Proposition}

\newtheorem{lemma}[theorem]{Lemma}

\theoremstyle{definition}
\newtheorem{definition}[theorem]{Definition}

\theoremstyle{remark}
\newtheorem{remark}[theorem]{Remark}

\theoremstyle{remark}
\newtheorem{example}[theorem]{Example}

\numberwithin{equation}{section}

\begin{document}

\bibliographystyle{amsplain}

\title[Recurrence analysis of uniform binary substitutions]
 {Recurrence analysis of uniform binary substitutions}

\author[M. Pol\'akov\'a]{Miroslava Pol\'akov\'a}
\address{
  Department of~Mathematics, Faculty of~Natural Sciences, Matej Bel University, Tajovsk\'eho~40,
  Bansk\'a Bystrica, Slovakia
}
\email{miroslava.sartorisova@umb.sk}

\author[V. \v Spitalsk\'y]{Vladim\'ir \v Spitalsk\'y}
\address{
  Beset, spol. s r.o. Jelenia 18, Bratislava, Slovakia;
  Department of~Mathematics, Faculty of~Natural Sciences, Matej Bel University, Tajovsk\'eho~40,
  Bansk\'a Bystrica, Slovakia
}
\email{vladimir.spitalsky@beset.sk, vladimir.spitalsky@umb.sk}



\begin{abstract}
	Recurrence rate, determinism, average line length, and entropy of line lengths 
	are measures of complexity in recurrence quantification analysis,
	that help to understand the structure, predictability and complexity of dynamical systems.
	In this paper, we provide explicit formulas for these characteristics for all binary substitutions of constant length.
	Moreover, we show that the determinism of such substitutions
	converges to one, which shows that these systems are highly predictable.
\end{abstract}
\maketitle
\thispagestyle{empty}

\section{Introduction}\label{S:intro}

Recurrence is an important element in analyzing and understanding
the behavior of systems evolving over time.
In 1987, Eckmann, Kamphorst, and Ruelle \cite{eckmann1987recurrence} introduced
recurrence plots, which provide a visual representation of recurrence patterns in dynamical systems.
If two points are close enough to each other (based on predefined embedding dimension $m$,
distance threshold $\eps$ and metric $\rho$),
a black dot is placed at the intersection of the corresponding times in the grid.
The structures in recurrence plots can provide valuable insights into predictability
of a system and reveal other aspects of its behavior.
Long diagonal lines typically indicate recurrent behavior and potential predictability of a system.

To provide a tool for quantifying and comparing these structures, recurrence quantification analysis (RQA) was introduced in 1992 by Zbilut and Webber \cite{zbilut1992embeddings}.
Most commonly used measures, including recurrence rate, determinism, average line length, and entropy of line lengths, are based on diagonal lines.
The recurrence rate ($\rr$) is the ratio of the recurrences to the total number of points in the recurrence plot.
It quantifies how often a dynamical system recurs to previously visited states.
Determinism ($\deter$) is closely related to the recurrence rate. It characterizes the percentage of recurrences in sufficiently long diagonal lines. High determinism values show signs of a more predictable system, while low values indicate more chaotic or random behavior.
Average line length ($\lavg$) is a quantitative measure that is defined as the average length
of diagonal lines within the recurrence plot.
A higher value suggests that, on average, the recurrent behavior of a system involves
longer sequences of revisiting similar states.
Entropy of line lengths ($\ent$) is the Shannon entropy of the distribution of diagonal line lengths in the recurrence plot.

Another measure for quantifying recurrences, not based on the recurrence plot,
is the correlation sum ($\corsum$).
It is the relative frequency of recurrences within the initial segment of the trajectory,
and is closely linked to the recurrence rate \cite[Proposition~1]{grendar2013strong}. 

In 2010, Faure and Lesne \cite{faure2010recurrence} introduced symbolic recurrence plots,
an extension of recurrence analysis to symbolic sequences.
An important advantage of a symbolic recurrence plot is that it
depends neither on the distance threshold nor on the embedding dimension.
The reason for this is the fact that every line of length $\ell$ in the recurrence plot with distance threshold $\eps = 2^{-h}$ and embedding dimension $m$ corresponds,
in a one-to-one way, to an $(\ell+h+m-2)$-line (with the same starting point) in the symbolic recurrence plot; for the details, see e.g.~Subsections~\ref{SUBS:gen-eps} and \ref{SUBS:gen-emb-dim}.

Over the years, various extensions and modifications have been developed \cite{marwan2014mathematical, webber2015recurrence}.
As can be seen in \cite{marwan2023trends}, recurrence analysis of dynamical systems continues to evolve, with new theoretical considerations and recurrence quantifiers being developed.

\medskip

Dynamical systems arising from substitutions are an important class of dynamical systems that have been studied in depth for many years (see for example \cite{gottschalk1955topological, martin1971substitution, martin1973minimal, michel1976stricte, dekking1978spectrum,  queffelec2010substitution}). 
They are widely used in many areas, with the capability to generate complex and interesting patterns in sequences and to model dependence on initial conditions.
Substitutions are often utilized as a method for generating symbolic sequences.

A substitution of constant length or uniform substitution
is a  map $\zeta$ from a finite set $\AAa$ (called an alphabet) to the set $\AAa^*$ of words over $A$, such that there exists $q\ge 2$ which is the length of ${\zeta(a)}$ for every $a \in \AAa$.
The language of the substitution $\zeta$ is
\begin{equation*}
	\lang_\zeta = \{w\in\AAa^*\colon w \text{ is a subword of some } \zeta^k(a)\}.
\end{equation*}
The substitution dynamical system induced by the substitution $\zeta$, or simply the $\zeta$-subshift, is the subshift $(X_\zeta,\sigma)$ defined by
\begin{equation*}
	X_\zeta = \{
	y\in\Sigma\colon
	y_{[0,n)} \in\lang_\zeta \text{ for every } n\ge 1
	\},
\end{equation*}
where $\Sigma=\AAa^{\NNN_0}$ is the compact metric space of all infinite sequences over $\AAa$,
and $\sigma\colon\Sigma\to\Sigma$ is the (right) shift.
The substitution $\zeta$ is called aperiodic if $X_\zeta$ contains a sequence
which is not $\sigma$-periodic, 
and is called primitive if there exists $k \geq 1$ such that 
$b$ is in $\zeta^k (a)$ for every $a,b \in \AAa$ \cite[Definitions~5.15 and 5.3]{queffelec2010substitution}.

If $\zeta$ is primitive, then the induced subshift $X_\zeta$ is strictly ergodic
(that is, minimal and uniquely ergodic) and
two cases can happen. First, if $\zeta$ is not aperiodic,
then every diagonal line
in the (infinite) symbolic recurrence plot has infinite length; 
thus the recurrence characteristics are obvious. 
Second, if $\zeta$ is aperiodic,
the following theorem gives formulas for recurrence characteristics
via densities of the sets $\kkk_\ell$,
which are the sets of starting points
of diagonal lines of length $\ell\in\NNN$ in the infinite symbolic recurrence plot 
(see Subsection~\ref{SUBS:density-kkk_ell} for the details).
This and the main result of \cite{polakova2023symbolic} yield to 
closed-form formulas, see Theorems~\ref{T:formulas} and \ref{T:formulas-ent},
and Section~\ref{S:examples} for examples.

\begin{theorem}[Recurrence characteristics of primitive substitutions]\label{T:formulas-via-infinite-sums}
	Let $\zeta$ be a primitive aperiodic binary substitution of constant length.
	Then, for every $y\in X_\zeta$, $m,\ell\in\NNN$ and $\eps=2^{-h}$ ($h\in\NNN$),
	\begin{eqnarray*}
		\linedens_\ell(y^{(m)},\infty,\eps) 
		&=&
		\dens(\kkk_{\ell'}),
	\\ 
		\lineDens_\ell(y^{(m)},\infty,\eps) 
		&=&
		\sum_{l\ge\ell'} \dens(\kkk_{l}),
	\\ 
		\rr_\ell(y^{(m)},\infty,\eps) 
		&=&
		\sum_{l\ge\ell'} [l- (m+h-2)] \dens(\kkk_{l}),
	\\ 
		\corsum_\ell(y^{(m)},\infty,\eps) 
		&=&
		\sum_{l\ge\ell'} [l- (m+h-\ell-1)] \dens(\kkk_{l}),
	\end{eqnarray*}
	where $\ell'=\ell+m+h-2$.
	The corresponding formulas for recurrence determinism and average line length readily follows.
	Moreover, if $x$ is a periodic point of $\zeta$ then
	\begin{eqnarray*}
		\ent_\ell(x^{(m)},\infty,\eps) 
		&=&
		\log \Big( \sum_{l\ge\ell'} \dens(\kkk_{l}) \Big)
		-
		\frac{1}{\sum_{l\ge\ell'} \dens(\kkk_{l})}
		\sum_{l\ge\ell'} \dens(\kkk_{l})\log \dens(\kkk_{l}).
	\end{eqnarray*}
\end{theorem}

On the other hand, if $\zeta$ is not primitive, then at least one of $\zeta(0), \zeta(1)$ 
is ``trivial'', that is, it belongs to $\{0^q,1^q\}$. Such a substitution
yields either to a finite subshift $X_\zeta=\{0^\infty,1^\infty\}$ (see Remark~\ref{REM:nonPrimitive}),
or to a proximal one with unique invariant measure (see Proposition~\ref{P:nonprimit-uniq-ergod}).
In both cases, the recurrence characteristics are trivial, as is shown in the following theorem.

\begin{theorem}[Recurrence characteristics of non-primitive substitutions]\label{THM:nonPrimitive}
	Let $\zeta$ be a non-primitive binary substitution of constant length.
	Then, for every $y\in X_\zeta$,  $m,\ell\in\NNN$ and $\eps>0$, 
	\begin{equation*}
		\corsum_\ell(y^{(m)},\infty,\eps)
		= \rr_\ell(y^{(m)},\infty,\eps)
		= \deter_\ell(y^{(m)},\infty,\eps)
		=1
	\end{equation*}
	and
	\begin{equation*}
		\lavg_\ell(y^{(m)},\infty,\eps) = \infty.
	\end{equation*}
\end{theorem}
It is interesting to note that, for non-primitive substitutions, 
we cannot derive recurrence characteristics from the densities of the sets $\kkk_\ell$
of (starting points of) diagonal lines. In fact, for a large class of such substitutions, 
the density $\dens(\kkk_\ell)$ is zero for every $\ell$ \cite[Theorem~1.3]{polakova2023symbolic}.
Notice also that, under the assumptions of Theorem~\ref{THM:nonPrimitive}, we do not know 
whether the entropy of line lengths $\ent_\ell(y^{(m)},\infty,\eps)$ is infinite.

Using the formulas from the previous two theorems we can prove that
the subshift induced by any uniform binary substitution is well predictable in the sense
that every trajectory of it has determinism close to $1$ as the distance threshold
$\eps$ is sufficiently close to $0$.

\begin{theorem}[Limit of determinism]\label{C:determinism}\label{T:determinism}
	Let $\zeta$ be a uniform binary substitution. Then,
	for every $y\in X_\zeta$, $m,\ell \in\NNN$ and $\eps > 0$, 
	\begin{equation}\label{EQ:determinism}
		\lim\limits_{\eps \to 0} \deter_\ell(y^{(m)}, \infty, \eps) = 1.
	\end{equation}
\end{theorem}

\medskip
The paper is organized as follows. In Section~\ref{S:prelim} we recall notions and results which are used in the remaining part of the paper. Particularly, in Subsection~\ref{SUBS:rqa-and-corsum} we correct a flaw from \cite[Proposition~1]{grendar2013strong}. 
In Section~\ref{S:primit} we deal with primitive substitutions and we prove Theorem~\ref{T:formulas-via-infinite-sums}. Examples of explicit formulas
for recurrence characteristics are given in Section~\ref{S:examples}. 
Non-primitive substitutions are described in Section~\ref{S:nonPrimit}, where
we also give a proof of Theorem~\ref{THM:nonPrimitive}. In the final Section~\ref{S:limOfDet}
we prove Theorem~\ref{T:determinism}.

\section{Preliminaries}\label{S:prelim}

The set of non-negative (positive) integers is denoted by $\NNN_0$ ($\NNN$).
The cardinality of a set $B$ is denoted by $\card B$.
We adopt the conventions that $[a, b] = \emptyset$ for $a > b$ and $[a, a] = \{a\}$.
When no confusion can arise, the set of consecutive
integers $\{ a, a+1, \dots, b-1 \}$ for any $a < b$ from $\NNN_0$ is denoted by $[a, b)$
or $[a,b-1]$;
if $a \geq b$ then $[a, b) = [a,b-1] = \emptyset$. 
For $B \subseteq \RRR^{k}$ and $c, d \in \RRR$ put $cB+d = \{ cb+d\colon b \in B \}$.
We adopt the convention that $0\log 0=0$.

Let $k \in \NNN$ and $M \subseteq \NNN_0^k$.
Then the \emph{upper} and \emph{lower (asymptotic) densities} of $M$ are given by
$$ 
	\bar \dens (M) = \limsup\limits_{n \to \infty} \frac{1}{n^k} \card \big(M \cap [0, n)^k\big) 
	\quad\text{and}\quad 
	\ubar{\dens}(M) = \liminf\limits_{n \to \infty} \frac{1}{n^k} \card \big(M \cap [0, n)^k \big).
$$
If $\ubar{\dens}(M) = \bar \dens (M)$ we say that the \emph{(asymptotic) density} $\dens(M)$ of $M$
exists and it is equal to this common value.
If densities of disjoint sets $M,N \subseteq \NNN_0^k$ exist, then
\begin{equation}\label{EQ:density-props}
	\dens(M\sqcup N)=\dens(M)+\dens(N)
	\quad\text{and}\quad
	\dens(aM+b)=a^{-k} \dens(M)
\end{equation}
for every integers $a>0$ and $b$. 
Further, for any $M \subseteq \mathbb{N}_0^k$ and $N \subseteq \mathbb{N}_0^l$, $\dens(M \times N) = \dens(M)\dens(N)$ provided $\dens(M)$ and $\dens(N)$ exist.
Recall also that the \emph{upper Banach density} 
$\dens^*(M)$ of a set $M\subseteq\NNN_0$ is defined as
\begin{equation*}
	\dens^*(M) = \limsup_{m<n,\, n-m\to\infty} \frac{1}{n-m}\card \big(M\cap [m,n)\big).
\end{equation*}

\subsection{The metric space of symbolic sequences}
An \emph{alphabet} $\AAa$ is a nonempty finite set; elements of $\AAa$ are called \emph{letters}.
The set of all words over $\AAa$ is
\begin{equation*}
	\AAa^{*}  = \bigcup_{\ell \ge 0} \AAa^\ell.
\end{equation*}
The set $\AAa^*$ endowed with concatenation is a free monoid.
Any member of $\AAa^\ell$ is called a \emph{word of length $\ell$}, or an \emph{$\ell$-word}, and is denoted by $w=w_0 w_1 \dots w_{\ell-1}$; here $w_i$ is the \emph{$i$-th letter} of $w$.
Length of a word $w$ will be denoted by $\abs{w}$.
The unique word of length $0$ is called the \emph{empty word} and is denoted by $o$.
A \emph{prefix} (\emph{suffix}) of a word $w=w_0 w_1 \dots w_{\ell-1}$ is any word
$w_{[0,k)}$ ($w_{[\ell-k,\ell)}$, respectively), where $0\le k\le \ell$; if $k<\ell$, the prefix (suffix) is
said to be \emph{proper}.

Put 
\begin{equation*}
	\Sigma=\AAa^{\NNN_0}.
\end{equation*}
Members $x = x_0 x_1 x_2 \dots $ of $\Sigma$ are called \emph{sequences}.
The \emph{subword of x of length $\ell\in\NNN_0$ starting at the index $i$}\footnote{We use the terms subword and index instead of factor and rank, respectively.} is an $\ell$-word $x_i x_{i+1} \dots x_{i+\ell-1}$ and will be denoted by $\word{x}{i}{i+\ell}$. The set of all subwords of $x$ is called the \emph{language of $x$} and
denoted by $\lang(x)$.

Metric $\rho$ on $\Sigma$ is defined for every $x, y \in \Sigma$ by $\rho(x, y) = 0$
if $x=y$ and 
\begin{equation}\label{EQ:rho-def}
	\rho(x, y) = 2^{-h}
	\qquad\text{if } x \neq y, 
	\quad\text{where } 
	h = \min \{ i \geq 0 \colon x_i \neq y_i \}.
\end{equation}
The pair $(\Sigma,\rho)$ is a compact metric space.
For every word $w\in\AAa^*$, the \emph{cylinder} $[w]$ is the clopen (that is, closed and open) set
$\{x\in\Sigma\colon x_{[0,\abs{w})}=w\}$.

A \emph{shift} is the map
$\sigma\colon \Sigma \to \Sigma $ defined by $\sigma(x_0 x_1 x_2 \ldots) = x_1 x_2 \ldots$.
For each nonempty closed $\sigma$-invariant subset $Y \subseteq \Sigma$,
the restriction of $(\Sigma, \sigma)$ to $Y$ is called a \emph{subshift}; 
to abbreviate, we will often say that $Y$ itself is a subshift.
Recall that $Y$ is a subshift if and only if there is  a set $\lang_Y\subseteq A^*$ (the so-called \emph{set of allowed words})
such that
\begin{equation*}
	y\in Y 
	\quad\text{if and only if}\quad
	\lang(y)\subseteq \lang_Y.
\end{equation*}
The closure of the orbit 
$\Orb_\sigma(x) = \{ \sigma^n(x) \colon n \geq 0 \}$ of any $x \in \Sigma$ defines a subshift,
as it is always a nonempty, closed and $\sigma$-invariant set; the language of it is equal to the language of $x$.

\subsection{Correlation sum}\label{S:corr-sum}
For $\ell\ge 1$, the Bowen metric $\rho_\ell$ on $\Sigma$ is given by
$$ 
	\rho_\ell (y,z) = \max \limits_{0 \leq i < \ell}
	\rho(\sigma^i(y), \sigma^i (z))
	\qquad
	\text{for every } y,z \in \Sigma.
$$

\begin{lemma}\label{LMM-emb-rho-ell-and-rho}\label{L:CorSum}
	Let $y,z \in \Sigma$, $\ell \in \NNN$, $h \in \NNN_0$. 
	\begin{enumerate}[label=(\alph*)]
		\item \label{LMM-rho-ell-and-rho-case-a}
		If $h \geq 1$, then $\rho_\ell(y,z) = 2^{-h}$ if and only if $\rho(y,z) = 2^{-h-\ell+1}$.
		\item \label{LMM-rho-ell-and-rho-case-b}
		If $h = 0$, then $\rho_\ell(y,z) = 2^{-0}$ if and only if $\rho(y,z) \geq 2^{-\ell+1}$.
	\end{enumerate}
\end{lemma}

\begin{proof}
	The proof is trivial, since
	\ref{LMM-rho-ell-and-rho-case-a} corresponds to $y_{[0, h+\ell-1)} = z_{[0, h+\ell-1)}$ and
		$y_{h+\ell-1} \neq z_{h+\ell-1}$,
	and \ref{LMM-rho-ell-and-rho-case-b} corresponds to $y_{[0, \ell)} \neq z_{[0, \ell)}$.
\end{proof}

In connection with correlation dimension \cite{grassberger1983characterization,grassberger1983measuring}
and correlation entropy \cite{takens1983invariants}, 
the so-called correlation sums were introduced.
Recall that, for $x\in\Sigma$, $\eps >0 $, $n \in \NNN$ and $\ell \in \NNN$, the \emph{correlation sum}
$\corsum_\ell(x, n, \eps)$ is given by
$$ 
	\corsum_\ell(x, n, \eps) 
	= 
	\frac{1}{n^2}\card \{ (i, j)\colon
	0 \leq i,j < n,\ \rho_\ell(\sigma^i(x), \sigma^j(x)) \leq \eps \}.
$$
By \cite{pesin1993rigorous}, if $\mu$ is an ergodic measure which preserves the shift $\sigma$, then
\begin{equation}\label{EQ:corr-integral}
	\lim_{n\to\infty} \corsum_\ell(x, n, \eps)  
	= 
	\mu\times\mu \{
		(y,z)\in \Sigma\times\Sigma\colon \rho_\ell(y,z)\le\eps
	\}
\end{equation}
for $\mu$-almost every $x\in \Sigma$ and for every $\eps>0$.\footnote{Note that, in \cite{pesin1993rigorous}, \eqref{EQ:corr-integral} is proved for \emph{all but countably many} $\eps>0$ (namely, for all points of continuity of the nondecreasing map $\varphi(\eps)=\mu\times\mu \{
		(y,z)\in \Sigma\times\Sigma\colon \rho_\ell(y,z)\le\eps
	\}$). However, due to
the fact that the metric $\rho_\ell$ attains only isolated positive values $2^{-h}$ ($h\in\NNN_0$), 
we can slightly increase $\eps$ to a continuity point $\eps'$ of $\varphi$ without affecting either side of \eqref{EQ:corr-integral}.}
Moreover, if $(X,\sigma)$ is a \emph{uniquely ergodic} subshift (that is, it has a unique invariant measure),
then \eqref{EQ:corr-integral} holds even for \emph{every} $x\in X$ and every $\eps>0$
(see e.g.~\cite[Proposition~3]{spitalsky2018local}).

\subsection{Recurrence plot} \label{SUBS:RecPlot}
A recurrence plot \cite{eckmann1987recurrence} visualizes trajectory of a dynamical system.
For a sequence $x\in\Sigma$, $n \in \NNN \cup \{ \infty \}$, $n \geq 2$ and $\eps >0$, the \emph{recurrence plot} $\recplot(x,n, \eps)$ is a square $n\times n$ matrix
such that, for $0 \leq i, j < n$,
$$ \recplot(x, n, \eps)_{i,j} = 
\begin{cases}
	1 &\text{  if } \rho(\sigma^i(x), \sigma^j(x)) \leq \eps, \\
	0 &\text{  otherwise}.
\end{cases}
$$
Every pair $(i,j)$ with $\recplot(x, n, \eps)_{i,j}=1$ is called a \emph{recurrence}.

Let $\recplot(x, n, \eps)$ be a recurrence plot and $\ell\ge 1$ be an integer.
A \emph{(diagonal) line of length $\ell$} (or, shortly, an \emph{$\ell$-line}) in $\recplot(x, n,\eps)$ is a triple $(i,j,\ell)$ of integers, where
\begin{itemize}
	\item $0 \leq i,j \leq n-\ell$ and $i\ne j$;
	\item $\recplot(x, n, \eps)_{i+h, j+h}  = 1$ for every $0 \leq h < \ell$;
	\item if $\min\{i,j\} > 0$, then $\recplot(x, n, \eps)_{i-1, j-1}  = 0$;
	\item if $\max\{i,j\} < n-\ell$, then $\recplot(x, n, \eps)_{i+\ell, j+\ell}  = 0$.
\end{itemize}
The pair $(i,j)$ is called the \emph{starting point}
and $\ell$ is called the \emph{length} of the line $(i,j,\ell)$.
If $\min\{i, j\} = 0$ we say that the line is \emph{$0$-boundary}.
Similarly,  if $\max\{i,j\} = n -\ell$, 
we say that the line is \emph{$n$-boundary}
(notice that if $n = \infty$ then no line is $n$-boundary; further, for $n$ finite,
a line can be both $0$-boundary and $n$-boundary).
Lines in $\recplot(x, n, \eps)$, which are neither $0$-boundary nor $n$-boundary,
are called \emph{inner lines}. 

In $\recplot(x, \infty, \eps)$ we analogously define also \emph{(diagonal) lines of length $\ell=\infty$}.
However, such lines occur only for eventually $\sigma$-periodic $x$, as is shown e.g.~in 
\cite[Proposition~2.1]{polakova2023symbolic}. 
Recall that 
$x$ is \emph{$\sigma$-periodic} if there is $p\in\NNN$ such that $x=(x_{[0,p)})^\infty$; the 
smallest such $p$ is called the \emph{period} of $x$.
Further, $x$ is \emph{eventually $\sigma$-periodic} if there is $h\in\NNN$ such that $\sigma^h(x)$ is $\sigma$-periodic.

\subsection{Recurrence characteristics} \label{SUBS:RQA}
Fix \emph{finite} $n\ge 2$ and $\ell \geq 1$. Let $\numberofdiag_{\ell}(x, n, \eps)$ denote the number of (possibly $0$-boundary or $n$-boundary) lines of length exactly $\ell$ in the recurrence plot $\recplot (x, n, \eps)$.
Denote the relative frequency of the starting points of lines of length exactly $\ell$ and those of length at least $\ell$ by
$$
	\linedens_\ell (x, n, \eps) = \dfrac{1}{n^2 - n} \numberofdiag_{\ell}(x, n, \eps)
	\qquad\text{and}\qquad
	\lineDens_\ell (x, n, \eps) = \sum_{l\ge\ell} \linedens_l (x, n, \eps).
$$
\emph{Recurrence rate}, \emph{determinism}, \emph{average line length} and \emph{entropy of line lengths} are given by
(see e.g.~\cite[Sec.~1.3.1]{marwan2014mathematical}; to abbreviate, on the right-hand sides we omit the arguments $x,n,\eps$)
\begin{eqnarray*}
	\rr_\ell  (x, n, \eps) &=& \sum_{l \geq \ell} l \linedens_l,
\\
	\deter_\ell(x, n, \eps)  &=& \frac{\rr_\ell}{\rr_1}\,,
\\
	\lavg_\ell(x, n, \eps)  &=& \frac{\rr_\ell}{\lineDens_\ell}\,,
\\	
	\ent_\ell(x, n, \eps)  &=& - \sum_{l\ge\ell} 
		\frac{\linedens_l}{\lineDens_\ell}
		\log \frac{\linedens_l}{\lineDens_\ell}
\ 
	= \ 
	\log \lineDens_\ell
	- \frac{1}{\lineDens_\ell}	
	\sum_{l\ge\ell} 
			\linedens_l		\log \linedens_l\,.
\end{eqnarray*}
(In the case when a denominator is zero, we leave the corresponding quantity undefined. Further,
in the final formula the convention $0\log 0=0$ is used.)
If limits for $n \to \infty$ exist, we denote them by $\linedens_{\ell} (x, \infty, \eps)$, $\rr_{\ell} (x, \infty, \eps)$, etc.

\subsection{Recurrence characteristics and correlation sum}\label{SUBS:rqa-and-corsum}
By \cite[(10) and Proposition~1]{grendar2013strong},
\begin{equation}\label{EQ:rr-and-corr-gms}
\begin{split}
	\corsum_\ell(x,n,\eps) &= \frac{1}{n^2}
	\Big[
		\sum_{l\ge\ell} (l-\ell+1)\numberofdiag_{l}(x,n,\eps) + n
	\Big],
\\
	\rr_\ell  (x, n, \eps) &= \frac{n}{n-1} 
	\Big[
		\ell\corsum_\ell(x,n,\eps) - (\ell-1)\corsum_{\ell+1}(x,n,\eps)
	\Big]
	 - \frac{1}{n-1}
\end{split}
\end{equation}
for every $\ell\in\NNN$, $n\ge 2$ and $\eps>0$\footnote{The differences between these formulas and those from \cite{grendar2013strong} are due to the fact that, in \cite{grendar2013strong}, the main diagonal of recurrence plot is not excluded from the analysis, but instead it is counted as a line of length $n$.}.
However, this formula assumes that $n$-boundary lines are excluded from the definition of recurrence characteristics and correlation sum, which is illustrated in the following example.

\begin{example}\label{EX:rr-and-corsum}
	Let $x=010\,111\,010\dots$ be a unique fixed point of the substitution $\zeta\colon A\to A^*$
	given by $\zeta(0)=010$ and $\zeta(1)=111$ (see Subsection~\ref{SS:subst}
	for the corresponding definitions). 
	Take $n=6$ and $\eps=1/2$.
	Then, in the upper part of $\recplot(x,n,\eps)$, there are two lines of length at least $2$: one $0$-boundary line $(0,2,2)$ and one $n$-boundary line $(3,4,2)$;
	thus $\numberofdiag_{2}=4$, $\numberofdiag_{\ell}=0$ for every $\ell>2$, and
	$\rr_{2} = (2\cdot 4)/(6^2-6)$.
	
	On the other hand, for any $0\le i < j < n$, $\rho(\sigma^i(x), \sigma^j(x))\le 1/4$
	if and only if $(i,j)\in\{(0,2),(1,5), (3,4)\}$, and $\rho(\sigma^i(x), \sigma^j(x))\le 1/8$
	if and only if $(i,j)=(1,5)$. Hence, by Lemma~\ref{L:CorSum},
	$\corsum_{2} = ({2\cdot 3+6})/{6^2}$ and $\corsum_{3} = ({2\cdot 1+6})/{6^2}$.
	Consequently, neither of the equations \eqref{EQ:rr-and-corr-gms} is true for $\ell=2$.
	However, if we exclude $n$-boundary lines from the recurrence plot, then 
	$\rr_{2}=({2\cdot 2+3\cdot 0})/({6^2-6})$, $\corsum_{2} = ({2\cdot 1+6})/{6^2}$
	and $\corsum_{3} = ({2\cdot 0+6})/{6^2}$.
	In this case, both formulas \eqref{EQ:rr-and-corr-gms} are true for $\ell=2$.
\end{example}

In the following propositions we give correct versions of \eqref{EQ:rr-and-corr-gms} and formula (11) from \cite{grendar2013strong}.
Notice that Theorems~4 and 5 from \cite{grendar2013strong} are not affected by this flaw and remain true as stated in \cite{grendar2013strong}. Since these results are true for every dynamical system, till the end of this subsection
we assume that $X=(X,\rho)$ is any metric space and $f\colon X\to X$ is any continuous map,
and that correlation sum, recurrence plot and recurrence characteristics are defined as in Subsections~\ref{S:corr-sum}--\ref{SUBS:RQA}, with $(\Sigma,\sigma)$ replaced by $(X,f)$.

\begin{proposition}\label{PROP:corsum-and-numberofdiag}
Let $x\in X$, $\ell\in \NNN$, $n\ge 2$ and $\eps>0$. Then
\begin{eqnarray*}
	\corsum_\ell(x,n,\eps) &=& \frac{1}{n^2}
	\Big[
		\sum_{l\ge\ell} (l-\ell+1)\numberofdiag_{l}(x,n,\eps)
		\ + n + \triangle_\ell
	\Big]\,,
\end{eqnarray*}
where 
\begin{equation*}
	0\le \triangle_\ell\le 2(\ell-1)(n-1)\,.
\end{equation*}
\end{proposition}
\begin{proof}
To abbreviate, we omit the arguments $x$, $n$ and $\eps$. 
Realize that, for any different integers $i,j\in[0,n)$,
$\rho_\ell(\sigma^i(x),\sigma^j(x))\le\eps$ implies
$\recplot_{ij}=1$ and so $(i,j)$ is contained in a line in $\recplot$.
Thus,
\begin{equation*}
	\corsum_\ell 
	=
	\frac{1}{n^2}
	\Big[
		n + \sum_{(i,j,l)\colon\text{line}} m_{i j l} 
	\Big],	 
\end{equation*}
where, for any line $(i,j,l)$ in $\recplot$,
\begin{equation*}
	m_{i j l}= \card\{
		k\in[0,l)\colon  \rho_\ell(\sigma^{i+k}(x),\sigma^{j+k}(x))\le\eps
	\}.
\end{equation*}
From the definitions of a line in a recurrence plot and of the metric $\rho_\ell$ 
we immediately have that 
\begin{itemize}
\item $m_{ijl}=\max\{l-\ell+1,0\}$ if the line $(i,j,l)$ is not $n$-boundary;
\item $\max\{l-\ell+1,0\}\le m_{ijl}\le l$ if the line $(i,j,l)$ is $n$-boundary.
\end{itemize}
Thus,
\begin{equation*}
	\corsum_\ell 
	=
	\frac{1}{n^2}
	\Big[
		n + \triangle_\ell + \sum_{l\ge\ell} (l-\ell+1)\numberofdiag_{l}(x,n,\eps)
	\Big],	 
\end{equation*}
where 
\begin{equation*}
	\triangle_\ell
	= 
	\sum_{(i,j,l)\colon n\text{-boundary}} \big(m_{i j l}-\max\{l-\ell+1,0\} \big).
\end{equation*}
Since the number of $n$-boundary lines is at most $2(n-1)$ and the difference of $m_{i j l}$ and $\max\{l-\ell+1,0\}$ belongs to $[0,\ell)$, we have that $0\le \triangle_\ell \le 2(\ell-1)(n-1)$.
\end{proof}

\begin{proposition}[Recurrence rate]\label{PROP:rr-and-corsum}
Let $x\in X$, $\ell\in \NNN$, $n\ge 2$ and $\eps>0$. Then
\begin{equation*}
	\rr_\ell  (x, n, \eps) 
	= \frac{n}{n-1} 
	\Big[
		\ell\corsum_\ell(x,n,\eps) - (\ell-1)\corsum_{\ell+1}(x,n,\eps) 
	\Big]
	- \frac{1}{n-1}
	+ \delta_\ell^\rr,
\end{equation*}
where 
\begin{equation*}
	\abs{\delta_\ell^\rr} \le \frac{2\ell(\ell-1)}{n}\,.
\end{equation*}
\end{proposition}

\begin{proof}
As above, we omit the arguments $x$, $n$ and $\eps$. By Proposition~\ref{PROP:corsum-and-numberofdiag},
\begin{equation*}
\begin{split}
	n^2\big(\ell\corsum_\ell - (\ell-1)\corsum_{\ell+1}\big)
	&=
	\sum_{l\ge\ell} [\ell(l-\ell+1)-(\ell-1)(l-\ell)] \numberofdiag_{l}
	\ + n + \triangle_\ell'
\\
	&=
	(n^2-n)\rr_{\ell}
	+ n + \triangle_\ell'
	\,,
\end{split}
\end{equation*}
where $\triangle_\ell'=\ell\triangle_\ell - (\ell-1)\triangle_{\ell+1}$.
Since $0\le \triangle_\ell \le 2(\ell-1)(n-1)$ and $0\le \triangle_{\ell+1} \le 2\ell(n-1)$,
we obtain
\begin{equation*}
	\abs{\delta_\ell^\rr}
	=
	\frac{\abs{\triangle_\ell'}}{n^2-n} 
	\le 
	\frac{2\ell(\ell-1)}{n}\,.
\end{equation*}
\end{proof}

\begin{proposition}[Average line length]\label{PROP:lavg-and-corsum}
Let $x\in X$, $\ell\in \NNN$, $n\ge 2$ and $\eps>0$. Then
\begin{equation*}
	\lineDens_\ell  (x, n, \eps) = \frac{n}{n-1} 
	\Big[
		\corsum_\ell(x,n,\eps) - \corsum_{\ell+1}(x,n,\eps) 
	\Big]
	+ \delta_\ell^\numberofdiag,
\end{equation*}
where 
\begin{equation*}
	\abs{\delta_\ell^\numberofdiag} \le \frac{2\ell}{n}\,.
\end{equation*}
Consequently, if $\lineDens_\ell(x,n,\eps)>0$ 
(that is, if there is a line of length at least $\ell$ in
the recurrence plot $\recplot(x,n,\eps)$), then
\begin{equation*}
	\lavg_\ell  (x, n, \eps) = 
	\frac
	{
		\frac{n}{n-1} 
		\Big[
			\ell\corsum_\ell - (\ell-1)\corsum_{\ell+1}
		\Big]
		-\frac{1}{n-1}+ \delta_\ell^\rr
	}
	{
		\frac{n}{n-1} 
		\Big[
			\corsum_\ell - \corsum_{\ell+1}
		\Big]
		+ \delta_\ell^\numberofdiag
	}
	\,.
\end{equation*}
\end{proposition}
\begin{proof}
By Proposition~\ref{PROP:corsum-and-numberofdiag},
\begin{equation*}
	n^2\big(\corsum_\ell - \corsum_{\ell+1}\big)
	=
	(n^2-n)\lineDens_{\ell}
	+ \triangle_\ell'
	\,,
\end{equation*}
where $\triangle_\ell'=\triangle_\ell - \triangle_{\ell+1}$.
Hence
\begin{equation*}
	\abs{\delta_\ell^\numberofdiag}
	=
	\frac{\abs{\triangle_\ell'}}{n^2-n} 
	\le 
	\frac{2\ell}{n}\,.
\end{equation*}
The rest follows from the definition of $\lavg_\ell$.
\end{proof}

We will also need the following formula which is in a sense ``inverse'' to
those from Propositions~\ref{PROP:rr-and-corsum} and \ref{PROP:lavg-and-corsum}.

\begin{proposition}[Correlation sum]\label{PROP:corsum-and-rr}
Let $x\in X$, $\ell\in \NNN$, $n\ge 2$ and $\eps>0$. Then
\begin{equation*}
	\corsum_\ell(x,n,\eps)
	=
	\frac{n-1}{n} 
	\Big[
		\rr_{\ell}(x,n,\eps) 
		 - (\ell-1)\lineDens_{\ell}(x,n,\eps) 
	\Big]
	+ \delta_\ell^\corsum,
\end{equation*}
where 
\begin{equation*}
	\frac{1}{n} \le \delta_\ell^\corsum < \frac{2\ell}{n}\,.
\end{equation*}
\end{proposition}
\begin{proof}
By Proposition~\ref{PROP:corsum-and-numberofdiag},
\begin{eqnarray*}
	\corsum_\ell
	&=&
	\frac{n^2-n}{n^2}\sum_{l\ge\ell} (l-\ell+1) \linedens_l
	\ +\  \frac{n+\triangle_\ell}{n^2}
\\
	&=&
	\frac{n-1}{n}
	\Big[
			\rr_{\ell}(x,n,\eps) 
			 - (\ell-1)\lineDens_{\ell}(x,n,\eps) 
	\Big]
	+ \frac{n+\triangle_\ell}{n^2}
\end{eqnarray*}
(we omitted the arguments $x$, $n$ and $\eps$), from which the proposition follows.
\end{proof}

By sending $n$ to infinity in Propositions~\ref{PROP:rr-and-corsum}--\ref{PROP:corsum-and-rr} we obtain the next theorem.

\begin{theorem}[{\cite[Theorems 4 and 5]{grendar2013strong}}]\label{THM:rr-and-lavg-via-corsum}
Let $x\in X$, $\ell\in \NNN$ and $\eps>0$. Assume that $\corsum_1(x,\infty,\eps)$,
$\corsum_\ell(x,\infty,\eps)$ and $\corsum_{\ell+1}(x,\infty,\eps)$ exist.
Then (on the right-hand sides, the arguments $x,n,\eps$ are omitted)
\begin{eqnarray*}
	\rr_\ell  (x, \infty, \eps) 
	&=& 
	\ell\corsum_\ell - (\ell-1)\corsum_{\ell+1},
\\
	\deter_\ell  (x, \infty, \eps) 
	&=& 
	\frac{\ell\corsum_\ell - (\ell-1)\corsum_{\ell+1}}
	     {\corsum_1}
	\,,
\end{eqnarray*}
and, provided that $\corsum_\ell(x,\infty,\eps)>0$,
\begin{equation*}
	\lavg_\ell  (x, \infty, \eps) 
	= 
	\ell 
	+ 
	\frac
		{\corsum_{\ell+1}}  
		{\corsum_{\ell} -\corsum_{\ell+1}}\,;
\end{equation*}
in the final formula, $\lavg_\ell  (x, \infty, \eps) =\infty$ provided
the denominator is zero. Further,
\begin{equation*}
	\corsum_\ell(x,\infty,\eps)
	=
	\rr_{\ell}
	 - (\ell-1)\lineDens_{\ell} .
\end{equation*}
\end{theorem}

\begin{remark}\label{REM:rr-and-lavg-via-corsum}
Note that, if $X$ is compact, the condition $\corsum_\ell(x,\infty,\eps)>0$ 
is always satisfied \cite[Lemma~8]{spitalsky2018local}.
Further, by \cite{pesin1993rigorous}, 
if $\mu$ is an ergodic measure of $(X,f)$ then, for every $\ell\in\NNN$,
$\corsum_\ell(x,\infty,\eps)$ exists for $\mu$-almost all $x\in X$ and for all
up to countably many $\eps>0$, and is equal to the correlation integral
$\mu\times\mu\{(y,z)\in X\times X\colon \rho_\ell(x,y)\le\eps\}$.
If $(X,f)$ is uniquely ergodic, then this is true for every $x\in X$
\cite[Proposition~3]{spitalsky2018local}.

These facts together with Theorem~\ref{THM:rr-and-lavg-via-corsum} imply that,
provided $X$ is compact and $(X,f)$ is {uniquely ergodic},
asymptotic recurrence rate, determinism and average line length do
not depend on $x\in X$. We do not know whether this is true also for the entropy of line lengths.
\end{remark}

\subsection{Recurrence plots and symbolic recurrence plots}\label{SUBS:gen-eps}
In this subsection we recall the way for obtaining recurrence characteristics 
of a sequence $x\in\Sigma$ for general $\eps$ 
from those calculated from symbolic recurrence plots \cite{faure2010recurrence},
that is, from recurrence plots of a symbolic sequence 
with $\eps_0=1/2$. We start with the following simple lemma.

\begin{lemma}\label{LMM:generalEps}
	Let $\ell\ge 1$ and $n \geq 2$.
	Then, for every different $0 \leq i,j < n$ and every $\eps = 2^{-h}$ ($h \in \NNN$),
	the following are equivalent:
	\begin{enumerate}[label=(\alph*)]
		\item \label{ITEM:eps-1} 
		$(i, j,\ell)$ is a line in $\recplot(x, n, \eps)$;
		
		\item \label{ITEM:eps-2} 
		$(i, j, \ell+h-1)$ is a line in $\recplot(x, n+h-1, \eps_0)$.	
	\end{enumerate}
\end{lemma}

\begin{proof}
	We will assume that $\min\{i,j\}>0$ and $\max\{i,j\}<n-\ell$; the other cases can be described analogously.
	
	Note that, by \eqref{EQ:rho-def} and the definition of recurrence plot,
	$\recplot(x,n+h-1,\eps_0)_{ij}=1$ if and only if $x_{i}=x_j$.
	Thus, $(i, j,\ell+h-1)$ is a line in $\recplot(x, n+h-1, \eps_0)$
	if and only if  
	$x_{[i,i+\ell+h-1)}=x_{[j,j+\ell+h-1)}$,
	$x_{i-1} \neq x_{j-1}$ and
	$x_{i+\ell+h-1} \neq x_{j+\ell+h-1}$.
		
	Analogously, $\recplot(x,n,\eps)_{ij}=1$ if and only if $x_{[i,i+h)}=x_{[j,j+h)}$.
	So,
	$(i, j,\ell)$ is a line in $\recplot(x, n, \eps)$
	if and only if 
	$x_{[i+k,i+k+h)}=x_{[j+k,j+k+h)}$ for every $k\in[0, \ell)$,
	$x_{[i-1,i-1+h)} \neq x_{[j-1,j-1+h)}$ and
	$x_{[i+\ell,i+\ell+h)} \neq x_{[j+\ell,j+\ell+h)}$.
	This is clearly equivalent to the fact that $(i, j,\ell+h-1)$ is a line in $\recplot(x, n+h-1, \eps_0)$.
\end{proof}

\begin{proposition}\label{PROP:recPlot-generalEps}
	Let $\ell\ge 1$, $n \geq 2$ and $\eps=2^{-h}$ ($h\in\NNN$).
	Then
	\begin{eqnarray*}
		L_\ell (x, n, \eps) 
		&=& 
		L_{\ell+h-1} (x, n+h-1, \eps_0),
	\\
		\linedens_\ell(x,n,\eps)  
		&=& 
		\theta_{nh} \linedens_{\ell+h-1}(x,n+h-1,\eps_0),
	\\
		\lineDens_\ell(x,n,\eps)  
		&=& 
		\theta_{nh} \lineDens_{\ell+h-1}(x,n+h-1,\eps_0),
	\\
		\rr_\ell (x, n, \eps) 
		&=& \theta_{nh} 
		\big(
			\rr_{\ell+h-1} (x, n+h-1, \eps_0) 
	\\
			&& \qquad- (h-1) \lineDens_{\ell+h-1}(x, n+h-1, \eps_0)
		\big),
	\end{eqnarray*}
	where $\theta_{nh}=((n+h-1)^2-(n+h-1)) / (n^2-n)$.
\end{proposition}
\begin{proof}
	The first equality follows from Lemma~\ref{LMM:generalEps}
	(notice that, in $\recplot(x, n+h-1, \eps_0)$, every line of length 
	$\ell+h-1$ must start in $[0,n)^2$). The second and third equalities are trivial
	consequences of the first one. The final equality follows from the definition and the second equality:
	\begin{eqnarray*}
		\rr_\ell (x, n, \eps)
		&=& \sum_{l\ge\ell} l\linedens_l(x,n,\eps)
	\\
		&=& \theta_{nh} \sum_{l\ge\ell} ((l+h-1)-(h-1)) \cdot\linedens_{l+h-1}(x,n+h-1,\eps_0)
	\\
		&=& \theta_{nh} 
		\big(
					\rr_{\ell+h-1} (x, n+h-1, \eps_0) 
	\\
		&&			\qquad- (h-1) \lineDens_{\ell+h-1}(x,n+h-1,\eps_0)
		\big).
	\end{eqnarray*}
\end{proof}

Analogous relations for other recurrence quantifiers (determinism, average line length and
entropy of line lengths) can be obtained using definitions and the preceding proposition.

\subsection{Embedded recurrence plots}\label{SUBS:gen-emb-dim}
In this subsection we recall that the analysis of embedded recurrence plots
for symbolic sequences
can be easily reduced to the analysis of nonembedded recurrence plots, as was mentioned
already in \cite{faure2010recurrence}. 
For completeness, we give all the details. 

Fix $m \in \NNN$.
Put $\Sigma^{(m)} = (\AAa^{m})^{\NNN_0}$ and 
define a metric $\rho^{(m)} $ on $\Sigma^{(m)}$ by 
\begin{equation*}
	\rho^{(m)} (y, z) =
	\begin{cases}
		2^{-h} &\text{ if } y \neq z ,  \text{ where }h= \min\{ i\ge 0\colon y_i \neq z_i \},
		\\
		0 &\text{ if }  y=z,
	\end{cases}
\end{equation*}
for every $y,z \in \Sigma^{(m)}$.

If $y\in\Sigma$ is a sequence over $\AAa$, then the \emph{embedded sequence} $y^{(m)}\in\Sigma^{(m)}$ 
is a sequence over $\AAa^{m}$ defined by
\begin{equation*}
	y^{(m)} = y_0^{(m)}  y_1^{(m)} \ldots = (y_0 y_1 \ldots y_{{m}-1})(y_1 y_2 \ldots y_{m}) \ldots,
	\qquad y_j^{(m)} = \word{y}{j}{j+m}.
\end{equation*}
For embedded sequences, the metric $\rho^{(m)}$ is tightly connected with the metric $\rho$.

\begin{lemma}\label{LMM-emb-rhod-and-rho}
	Let $y,z \in \Sigma$, $m \in \NNN$ and $h \in \NNN_0$. 
	\begin{enumerate}[label=(\alph*)]
		\item \label{LMM-emb-case-a} If $h \geq 1$, then $\rho^{(m)}(y^{(m)}, z^{(m)}) = 2^{-h}$ if and only if
		$ \rho(y,z) = 2^{-h-m+1}$.
		\item \label{LMM-emb-case-b} If $h =0 $, then $\rho^{(m)}(y^{(m)}, z^{(m)}) = 2^{-0}$ if and only if
		$ \rho(y,z) \geq 2^{-m+1}$.
	\end{enumerate}
\end{lemma}

\begin{proof}
	Case \ref{LMM-emb-case-a} corresponds to $y_{[0, h+m-1)} = z_{[0, h+m-1)}$ and
	$y_{h+m-1} \neq z_{h+m-1}$.
	Case \ref{LMM-emb-case-b} corresponds to $y_{[0, m)} \neq z_{[0, m)}$.
\end{proof}

The \emph{shift} $\sigma^{(m)}\colon \Sigma^{(m)} \to \Sigma^{(m)}$ is defined by
$$
	\sigma^{(m)}(y) = \sigma^{(m)}(y_0 y_1 y_2 \dots) =  y_1 y_2 \dots
$$
for $y \in \Sigma^{(m)}$. If $\pi^{(m)}\colon\Sigma\to\Sigma^{(m)}$ denotes the map sending every $y\in\Sigma$
to its embedded sequence $y^{(m)}\in\Sigma^{(m)}$, then clearly (see Figure~\ref{FIG:cd-shift-and-embedding})
\begin{equation}\label{EQ:cd-shift-and-embedding}
	\pi^{(m)}\circ\sigma = \sigma^{(m)}\circ \pi^{(m)}.
\end{equation}

\begin{figure}[!htb]\centering
	\begin{center}
		$\begin{CD}
		\Sigma     @>\sigma>>  \Sigma\\
		@V\pi^{(m)}VV        @VV\pi^{(m)}V\\
		\pi^{(m)}(\Sigma)    @>\sigma^{(m)}>>  \pi^{(m)}(\Sigma)
		\end{CD}$
		\caption{Commutativity of shift and embedding}\label{FIG:cd-shift-and-embedding}		
	\end{center}
\end{figure}

For every $\ell \in \NNN$, define the Bowen metric $\rho_\ell^{(m)}$ on $\Sigma^{(m)}$ by
$$
	\rho_\ell^{(m)}(y,z) = \max \limits_{0 \leq i < \ell}
	\rho^{(m)}\left(
		\bigl(\sigma^{(m)}\bigl)^i(y),   
		\bigl(\sigma^{(m)}\bigl)^i (z)
	\right)
$$
for every $y,z\in\Sigma^{(m)}$.
The following lemma can be proved analogously as Lemma~\ref{LMM-emb-rho-ell-and-rho}.
\begin{lemma}\label{LMM-emb-rhod-ell-and-rhod}
	Let $m\in\NNN$, $y,z \in \Sigma^{(m)}$, $\ell \in \NNN$ and $h \in \NNN_0$.
	\begin{enumerate}[label=(\alph*)]
		\item \label{LMM-emb-case-a-2}
		If $h \geq 1$, then $\rho^{(m)}_\ell(y,z) = 2^{-h}$ if and only if $\rho^{(m)}(y,z) = 2^{-h-\ell+1}$.
		\item \label{LMM-emb-case-b-2}
		If $h = 0$, then $\rho^{(m)}_\ell(y,z) = 2^{-0}$ if and only if $\rho^{(m)}(y,z) \geq 2^{-\ell+1}$.
	\end{enumerate}
\end{lemma}
As a consequence of this lemma and Lemmas~\ref{LMM-emb-rhod-and-rho} and \ref{LMM-emb-rho-ell-and-rho}
we obtain that, for every $\eps\in(0,1)$, $y,z\in\Sigma$ and $m,\ell\in\NNN$,
\begin{equation}\label{EQ:rho-m-ell}
	\rho_\ell^{(m)} (y^{(m)}, z^{(m)}) \leq \eps  
	\qquad\text{if and only if}\qquad
	\rho_\ell (y, z) \leq 2^{-m+1}\eps.
\end{equation}

Fix any sequence $x\in\Sigma$.
Clearly, the definitions of correlation sums, recurrence plots and recurrence characteristics can be 
applied also to the embedded sequence $x^{(m)}$. Notice that, for $\eps=2^{-h}$ $(h\in\NNN)$, $\recplot(x, n, \eps)$ depends on $x_{[0, n+h-1)}$
and $\recplot(x^{(m)}, n, \eps)$ depends on $x_{[0, n+h+m-2)}$.

The following two results immediately follow from \eqref{EQ:rho-m-ell} and \eqref{EQ:cd-shift-and-embedding}.

\begin{proposition}\label{PROP:embeddedCorSum}
	Let $x \in \Sigma$, $m,n,\ell\in\NNN$ and $\eps \in(0,1)$. Then
	$$ 
		\corsum_\ell(x^{(m)},n,\eps) = \corsum_\ell (x, n, 2^{-m+1}\eps).
	$$
\end{proposition}

\begin{proposition}\label{PROP:embeddedRecPlot}
	Let $x\in\Sigma$, $m,n\in \NNN$, ($n \geq 2$) and $\eps \in (0,1)$.
	Then
	\begin{equation*}
			\recplot(x^{(m)}, n, \eps) = \recplot(x, n, 2^{-m+1}\eps).
	\end{equation*}
	Consequently, the corresponding recurrence characteristics obtained from these two recurrence plots are equal.
\end{proposition}

\subsection{Substitutions}\label{SS:subst}
From now on we restrict our attention to the binary alphabet
$\AAa=\{0,1\}$. Put $\bar{0} = 1$ and $\bar{1} = 0$. 
A map $\zeta\colon \AAa\to\AAa^*$ is said to be 
a \emph{substitution of constant length} or \emph{uniform substitution}
provided there is $q\ge 2$ such that $\abs{\zeta(0)}=\abs{\zeta(1)}=q$.
So $\zeta$ maps each letter ($0$ and $1$) to a word of length $q$ ($\zeta(0)$ and $\zeta(1)$, respectively).

The substitution $\zeta$ induces a morphism (denoted also by $\zeta$) of the monoid $\AAa^*$
by putting $\zeta(o) = o$ and $\zeta(w) = \zeta(w_0) \zeta(w_1) \dots \zeta(w_{\ell-1})$
for any nonempty $\ell$-word $w = w_0 w_1 \dots w_{\ell-1}$.
The iterates $\zeta^k (k \geq 1)$ of $\zeta$ are defined inductively by $\zeta^1 = \zeta$
and $\zeta^k = \zeta \circ\zeta^{k-1}$ for $k\ge 2$.
Further, $\zeta$ induces a map
(again denoted by $\zeta$) from $\Sigma=\AAa^{\NNN_0}$ to $\Sigma$ by
$\zeta(x)=\zeta(x_0)\zeta(x_1)\dots$ for $ x=(x_n)_{n\in\NNN_0} \in \Sigma$.
Since no confusion can arise, all these maps will be simply called \emph{substitution} $\zeta$.

The \emph{language} of the substitution $\zeta$ is
\begin{equation*}
	\lang_\zeta = \{w\in\AAa^*\colon w \text{ is a subword of some } \zeta^k(a)\}.
\end{equation*}
The substitution $\zeta$ defines the \emph{substitution dynamical system}, which is the subshift $(X_\zeta,\sigma)$ with
\begin{equation*}
	X_\zeta = \{
		y\in\Sigma\colon
		y_{[0,n)} \in\lang_\zeta \text{ for every } n\ge 1
	\};
\end{equation*}
that is, $X_\zeta$ is the unique subshift with the language $\lang_\zeta$.
The substitution $\zeta$ is called \emph{aperiodic} if $X_\zeta$ contains a sequence
which is not $\sigma$-periodic \cite[Definition~5.15]{queffelec2010substitution}.
The substitution $\zeta$ is \emph{primitive} if there exists $k \geq 1$ such that, for every $a,b \in \AAa$,
$b$ is in $\zeta^k (a)$ \cite[Definition~5.3]{queffelec2010substitution}.

Assume that $\zeta$ is a primitive aperiodic binary substitution of constant-length $q$ such that
\begin{equation}\label{assumpt-start-with-0}
	\zeta(0) \text{ starts with letter } 0.
\end{equation}
Equivalently, by \cite{seebold1988periodicity}, we assume that $\zeta$ is a
binary substitution of constant-length $q$ such that
\begin{itemize}
\item $\zeta(0)$ starts with $0$ and contains $1$;
\item $\zeta(1)\ne\zeta(0)$ and contains $0$;
\item if $q=2s+1$ is odd, then $\zeta(0)\ne (01)^s0$ or $\zeta(1)\ne (10)^s1$.
\end{itemize}
Clearly, any such substitution $\zeta$ is \emph{injective} (or \emph{one-to-one on the alphabet}), 
that is, $\zeta (0) \neq \zeta (1)$.
Further, there is a unique fixed point $x \in\Sigma$ of $\zeta$ starting with $0$ \cite[page~126]{queffelec2010substitution}; we often write $x=\zeta^\infty(0)$. Obviously,
if we write $\zeta(0)=0v$ (where $v\in\AAa^*$ is of length $q-1$) then
\begin{equation*}
	x = 0v\zeta(v)\zeta^2(v)\dots\zeta^k(v)\dots
\end{equation*}
The subshift $X_\zeta$ is equal to the $\sigma$-orbit closure of the sequence $x$
\cite[Proposition~5.5]{queffelec2010substitution};
hence, $\lang_\zeta$ is equal to the language $\lang(x)$ of $x$.
Finally, by \cite{michel1976stricte} (see also \cite[Theorem~5.6 and Proposition~5.5]{queffelec2010substitution}),
the subshift $(X_\zeta, \sigma)$ is \emph{strictly ergodic}, that is, minimal and uniquely ergodic.

\subsection{Recognizability}
In this subsection, $\zeta$ is a primitive aperiodic binary substitution of constant length $q$
(hence injective) satisfying \eqref{assumpt-start-with-0} 
and $x=\zeta^\infty(0)$ is a unique fixed point of $\zeta$ starting with $0$.

By Mentzen (1989) and \cite{apparicio1999reconnaissabilite}
(see also \cite[Proposition~5.14]{queffelec2010substitution} or \cite[Theorem~4.31]{bruin2022topological}), 
the substitution $\zeta$ is \emph{(one-sided) recognizable}; that is,
there is an integer $K>0$ such that, for every $i,j\in\NNN_0$ with $x_{[i,i+K]}=x_{[j,j+K]}$,
$i$ is a multiple of $q$ if and only if $j$ is a multiple of $q$.
The smallest integer $K$ with this property is called
the \emph{recognizability index of $\zeta$}.

We say that a word $w\in A^*$ is \emph{recognizable} if there is a (unique) integer $p_w \in [0, q)$ such that $\{i \in \NNN_0 \colon \word{x}{i}{i+\abs{w}} = w \} \subseteq q \NNN_0 + p_w$.
For the connection between recognizability of a word and uniqueness of its $1$-cutting (as defined in \cite[p.~210]{fogg2002substitutions}), see \cite[Appendix~A]{polakova2023symbolic}.
The main fact states that a word $w\in\lang_\zeta$ of length at least $q$ 
is recognizable if and only if it has a unique $1$-cutting.
Further, $K$ is the recognizability index of $\zeta$ if and only if
$K$ is the smallest positive integer such that every word $w=x_{[hq,hq+K]}$ ($h\in\NNN_0$) 
is recognizable.

\begin{definition}\label{DEF:alpha-and-beta}
	Let $\alpha, \beta \in \NNN_0$ and $\recog\in\NNN$ be defined as follows:

	\begin{enumerate}[label=(\alph*)]
	\item 
	$\alpha$ is the length of the longest common prefix of $\zeta(0), \zeta(1)$:
	$$ 
		\word{\zeta(0)}{0}{\alpha} = \word{\zeta(1)}{0}{\alpha},
		\quad 
		\zeta(0)_\alpha \neq \zeta(1)_\alpha;
	$$
	
	\item
	$\beta$ is the length of the longest common suffix of $\zeta(0), \zeta(1)$:
	$$ 
		\word{\zeta(0)}{q-\beta}{q} = \word{\zeta(1)}{q-\beta}{q},
		\quad 
		\zeta(0)_{q-\beta-1} \neq \zeta(1)_{q-\beta-1};
	$$
	
	\item 
	$\recog$ is the smallest integer such that $\recog > \alpha+\beta$ and every word $w\in\lang_\zeta$
	of length (at least) $\recog$ is recognizable.
	\end{enumerate}
\end{definition}

Since $\zeta (0) \neq \zeta (1)$, both $\alpha, \beta$ are well-defined and 
$\alpha + \beta \le q-1$. Further, by the mentioned result of Mentzen and Apparicio,
also $\recog$ is well-defined (in fact, $K+1\le \recog\le K+q$, where $K$ is the 
recognizability index of $\zeta$).

\subsection{Density of the set $\kkk_\ell$ of (starting points of) inner lines in symbolic recurrence plot}\label{SUBS:density-kkk_ell}

Let $\zeta$ be a primitive aperiodic binary substitution of constant length $q$ satisfying 
\eqref{assumpt-start-with-0}, and $x=\zeta^\infty(0)$. 
For $\ell\in\NNN$, denote by $\kkk_\ell$ the set of starting points
of inner lines in $\recplot(x,\infty,\eps_0)$:
\begin{equation}\label{EQ:Kell}
	\kkk_\ell
	=
    \{ (i,j) \in \NNN^2\colon 
    	(i,j,\ell) 
		\text{ is an inner line in } \recplot(x,\infty, \eps_0) 
	\}.
\end{equation}

\begin{theorem}[{\cite[Theorem~1.2]{polakova2023symbolic}}]\label{THM:density}
    For any $\ell\in\NNN$, the following is true:
	\begin{enumerate}
		\item \label{Case1-in-THM:density} 
		$\kkk_\ell = \emptyset$ if and only if $\dens(\kkk_\ell) = 0$;

		\item \label{Case3-in-THM:density} 
		if $\ell \geq \recog$ and 
		$\kkk_\ell \neq \emptyset$, there are unique positive integers $\ell_0$ and $k$ such that
		$\ell_0 < \recog\le q\ell_0+\alpha+\beta$, $\ell = q^k \ell_0 + c(q^k-1) $ and
		$$ 
		\dens (\kkk_\ell)
		= q^{-2k}\dens(\kkk_{\ell_0}) 
		= \frac{(\ell_0 + c)^2 }{(\ell + c)^2} \dens(\kkk_{\ell_0}) ,
		$$
		where $c = (\alpha+\beta)/(q-1) \in [ 0, 1]$.	
	\end{enumerate}
	Consequently, the set of all integers $\ell$ with $\kkk_\ell\ne\emptyset$ is a zero density subset of $\NNN_0$.
\end{theorem}

\section{Recurrence quantification analysis for primitive substitutions; proof of Theorem~\ref{T:formulas-via-infinite-sums}}\label{S:primit}

Till the end of this section assume that $\zeta$ is a primitive aperiodic binary
substitution of constant length $q$ satisfying \eqref{assumpt-start-with-0},
and $x=\zeta^\infty(0)$.
We derive explicit formulas for recurrence rate, determinism, average line length
and entropy of line lengths based on the densities of the sets $\kkk_\ell$
of (starting points of) inner lines in symbolic recurrence plot $\recplot(x,\infty,\eps_0)$, 
see Theorems~\ref{T:formulas-via-infinite-sums}, \ref{T:formulas} and \ref{T:formulas-ent}. To this end, we first prove some technical lemmas.

\subsection{Auxiliary lemmas}\label{SUBSECTION:technical-lemmas}
For every $\ell\ge 1$ and $n\ge 2$ define
$$
	\delta_\ell(n) = \dfrac{\card (\kkk_\ell \cap [1, n)^2)}{n^{2} - n} \,.
$$

\begin{lemma}\label{LMM:ohranicenie-zhora}
	There are constants $0<\tilde{c}_0\le 1\le \tilde{c}$ such that, for every $\ell \geq 1$ and $n \geq 2$
	with $\kkk_\ell\cap[1,n)^2\ne\emptyset$,
	$$
		\dfrac{\tilde{c}_0}{\ell^2}
		\leq
		\delta_\ell(n) 
		\leq 
		\dfrac{\tilde{c}}{\ell^2} \,.
	$$
\end{lemma}

\begin{proof}
	By Theorem~\ref{THM:density}, the density of $\kkk_\ell$ exists and is positive for every $\ell$ such that $\kkk_\ell\ne\emptyset$. Hence,
	for every $\ell$ with $\kkk_\ell\ne\emptyset$ we can find $\tilde{c}_\ell\ge\tilde{c}_{0,\ell}>0$
	such that ${\tilde{c}_{0,\ell}}/{\ell^2}  \leq \delta_\ell(n) \leq {\tilde{c}_\ell}/{\ell^2}$
	for every $n\ge 2$ with $\kkk_\ell\cap[1,n)^2\ne\emptyset$. Thus, to prove the lemma, we may assume that $\ell\ge\recog$ and $\kkk_\ell\ne\emptyset$.

	Let $k \in \NNN$ and $\ell_0 < \recog$ be from Theorem~\ref{THM:density}.
	Fix any $n \geq 2$. Then, by Theorem~\ref{THM:density},
	\begin{equation}\label{EQ:LMM-ohranicenie-zhora1}
		\card (\kkk_\ell \cap [1, n)^2) = \card (\kkk_{\ell_0} \cap [1, z)^2)
		\leq z^2,
	\end{equation}
	where
	$$ 
		z 
		= 
		\frac{n}{q^k} + \frac{\beta (q^k-1)}{(q-1)q^k} 
		\ \in \ 
		\Big[
			\frac{n}{q^k}, \frac{n}{q^k} + 1
		\Big).  
	$$
	
	Assume first that $n > q^k$. Then $z < 2n/q^k$ and so
	$$
		\card (\kkk_\ell \cap [1, n)^2) < \frac{4n^2}{q^{2k}}.
	$$
	By Theorem~\ref{THM:density}, $\ell = q^k(\ell_0 + c) -c$, where $c = (\alpha + \beta)/(q-1)$; notice that $0 \leq c \leq 1$. Hence $q^k = (\ell+c)/(\ell_0 + c)$. This, together with the fact that $\ell_0 < \recog$, gives
	$\card (\kkk_\ell \cap [1, n)^2)   <    4 \recog^2 n^2/\ell^2$.
	Since $n > q^k \geq 2$, we have $n<n^2/2$ and so
	\begin{equation}\label{EQ:LMM-ohranicenie-zhora2}
		\delta_\ell(n)
		<
		\dfrac{8\recog^2}{\ell^2}.
	\end{equation}
	
	Assume now that $n \leq q^k$. Then $z < 2$ and
	$\kkk_{\ell_0} \cap [1, z)^2 \subseteq \kkk_{\ell_0} \cap \{(1,1)\} = \emptyset$.
	So, by \eqref{EQ:LMM-ohranicenie-zhora1}, 
	$\kkk_\ell\cap[1,n)^2=\emptyset$ and \eqref{EQ:LMM-ohranicenie-zhora2} is trivially true.
	
	We have proved that \eqref{EQ:LMM-ohranicenie-zhora2} is true
	for every $\ell\ge\recog$ and every $n\ge 2$, from which the existence of $\tilde{c}$
	immediately follows.

	\medskip
	To show the existence of $\tilde{c}_0$, denote by $\LLl_0$ the set of all integers $1\le l_0<\recog$
	such that $\kkk_{l_0}\ne\emptyset$. Fix positive $\gamma' < (1/4)\min\{\dens(\kkk_{l_0})\colon l_0\in\LLl_0\}$.
	Since $\lim_n \delta_{l_0}(n)=\dens(\kkk_{l_0})$ for every $l_0\in\LLl_0$, 
	there is $m_0$ such that $\delta_{l_0}(m)\ge 4\gamma'$ 
	for every integer $m\ge m_0$ and every $l_0\in \LLl_0$.
	We may assume that $m_0\ge 4$; hence, for every $l_0\in\LLl_0$ and every \emph{real} $s\ge m_0$,
	\begin{equation}\label{LMM:ohranicenie-zhora-proof1}
	\begin{split}
		\card (\kkk_{l_0} \cap [1, s)^2)
		&\ge
		\card (\kkk_{l_0} \cap [1, \lfloor s \rfloor)^2)
		\ge
		4\gamma'(\lfloor s \rfloor^2 - \lfloor s \rfloor)
		\ge 
		2\gamma' \lfloor s \rfloor^2
	\\		
		&>
		2\gamma' (s-1)^2
		\ge 
		\gamma' s^2.
	\end{split}
	\end{equation}
	Put $\gamma=\min\{\gamma', 1/m_0^2\}$. 
	By combining \eqref{LMM:ohranicenie-zhora-proof1} with the fact that $1> \gamma s^2$ for every $1\le s<m_0$,
	we obtain that, for every $l_0\in\LLl_0$ and every real $s\ge 1$, 
	\begin{equation*}
		\kkk_{l_0}\cap[1,s)^2\ne\emptyset
		\quad\implies\quad
		\card (\kkk_{l_0} \cap [1, s)^2) 
		>
		\gamma s^2.
	\end{equation*}
	This together with \eqref{EQ:LMM-ohranicenie-zhora1} and the fact that $q^k=(\ell+c)/(\ell_0+c)$ yield
	\begin{equation*}
		\card (\kkk_\ell \cap [1, n)^2)
		=
		\card (\kkk_{\ell_0} \cap [1, z)^2)
		> 
		\gamma z^2
		\ge
		\frac{\gamma n^2}{q^{2k}}
		=
		\frac{\gamma n^2(\ell_0+c)^2}{(\ell+c)^2}		
		\ge 
		\frac{\gamma n^2}{4\ell^2}		
	\end{equation*}
	for every $\ell\ge\recog$ and every integer $n>0$ such that
	$\kkk_\ell \cap [1, n)^2\ne\emptyset$. This immediately implies the existence of $\tilde{c}_0$.
\end{proof}

\begin{lemma}\label{LMM:gama-is-finite}
	Let $\gamma = (\gamma_\ell)_{\ell \geq 1}$ be defined by 
	\begin{equation*}
		\gamma_\ell = 
		\begin{cases}
			{\tilde{c}}/{\ell}   &\text{if } \kkk_\ell \neq \emptyset,
			\\
			0      &\text{otherwise},
		\end{cases} 
	\end{equation*}
	where $\tilde{c}$ is the constant
	from Lemma~\ref{LMM:ohranicenie-zhora}. Then
	$$\sum_{\ell = 1}^{\infty} \gamma_\ell < \infty. $$
\end{lemma}

\begin{proof}
	Put $\mathcal{L} = \{ \ell \in\NNN \colon \kkk_\ell \neq \emptyset \}$.
	Fix any $\ell \geq \recog$ from $\mathcal{L}$.
	By Theorem~\ref{THM:density}, there exist unique $k_\ell \geq 1$
	and $\ell_{0,\ell} <\recog$ such that $ \ell = q^{k_\ell} \big(\ell_{0,\ell} + c\big)-c$,
	where $c = (\alpha+\beta)/(q-1)$.
	Thus,
	$$ 
		\gamma_\ell =
		\frac{\tilde{c}}{q^{k_\ell}(\ell_{0,\ell}+c)-c} 
		\le
		q^{-k_\ell}  \gamma_{\ell_{0,\ell}} \, .
	$$
	Further, if $\ell \neq \ell'$ are such that $k_\ell = k_{\ell'}$,
	then $\ell_{0,\ell} \neq \ell_{0,\ell'}$.
	These two facts imply
	$$ 
		\sum_{\ell = 1}^{\infty} \gamma_\ell
		= 
		\sum_{\ell < \recog} \gamma_\ell +\sum_{k=1}^{\infty} 
		\sum_{\substack{\ell \geq \recog, \, \ell \in \mathcal{L} \\ k_\ell = k}} \gamma_\ell  
		\le
		\Gamma + \sum_{k=1}^{\infty}q^{-k} \Gamma 
		= 
		\frac{q}{q-1}\Gamma 
		< 
		\infty,
	$$
	where $\Gamma = \sum_{\ell < \recog} \gamma_\ell$.
\end{proof}

\begin{lemma}\label{LMM:moore-osgood}	
	For every integer $\ell \geq 1$,
	\begin{equation*}
		\lim\limits_{n \to \infty} \sum_{l \geq \ell} l\delta_{l}(n)
		=
		\sum_{l=\ell}^{\infty}l \dens(\kkk_l).	
	\end{equation*}
\end{lemma}

\begin{proof}
Since $\lim_n \delta_l(n) = \dens(\kkk_l)$ by the definition of density,
we have
$$ 
	\lim\limits_{n \to \infty} \sum_{l = \ell}^{k} l \delta_l(n) 
	= 
	\sum_{l = \ell}^{k} l \dens (\kkk_l)
$$
for every finite $k\ge\ell$.
Recall the definition of $\gamma_\ell$ from Lemma~\ref{LMM:gama-is-finite}.
By Lemma~\ref{LMM:ohranicenie-zhora},
$$
	0 
	\leq 
	l\delta_l(n) 
	\leq 
	\gamma_l 
	\qquad\text{for every } l \geq \ell.
$$
Lemma~\ref{LMM:gama-is-finite} and Weierstrass M-test yield
$$ 
	\lim\limits_{k \to \infty} \sum_{l = \ell}^{k} l \delta_l(n) 
	= 
	\sum_{l = \ell}^{\infty} l \delta_l(n) 
	\qquad\text{uniformly in }n.
$$
Now the lemma follows by Moore-Osgood theorem (see e.g. \cite[p.140]{taylor1985general}).
\end{proof}

\begin{lemma}\label{LMM:moore-osgood-linedens}	
	For every integer $\ell \geq 1$,
	\begin{equation*}
		\lim\limits_{n \to \infty} \delta_{\ell}(n)
		=
		\dens(\kkk_\ell)	
	\quad\text{and}\quad
		\lim\limits_{n \to \infty} \sum_{l=\ell}^{\infty} \delta_{l}(n)
		=
		\sum_{l=\ell}^{\infty} \dens(\kkk_l).	
	\end{equation*}
\end{lemma}
\begin{proof}
The first equality follows from the definition of the density.
The proof of the second equality is the same as that of Lemma~\ref{LMM:moore-osgood},
just instead of $l\delta_l(n)$ write $\delta_l(n)$ and instead of $l \dens (\kkk_l)$
write $\dens (\kkk_l)$, and use that $0\le\delta_l(n)\le\gamma_l/l \le\gamma_l$ for every $l\ge\ell$.
\end{proof}

\begin{lemma}\label{LMM:moore-osgood-ent}	
	For every integer $\ell \geq 1$,
	\begin{equation*}
		\lim\limits_{n \to \infty} \sum_{l=\ell}^{\infty} \delta_{l}(n)\log\delta_{l}(n)
		=
		\sum_{l=\ell}^{\infty} \dens(\kkk_l)\log \dens(\kkk_l).
	\end{equation*}
\end{lemma}
\begin{proof}
The proof is again the same as that of Lemma~\ref{LMM:moore-osgood},
just instead of $l\delta_l(n)$ write $-\delta_l(n)\log \delta_l(n)$ and instead of $l \dens (\kkk_l)$
write $-\dens (\kkk_l)\log \dens (\kkk_l)$, and use that,
for every $l\ge\ell$,
\begin{equation*}
	0
	\le
	-\delta_l(n)\log \delta_l(n)
	\le
	\frac{\tilde{c} (2\log l - \log \tilde{c}_0)}{l^2} \,,
\end{equation*}
where $\tilde{c}$ and $\tilde{c}_0$ are constants from Lemma~\ref{LMM:ohranicenie-zhora}
(notice that the last inequality is satisfied also in the case when $\delta_l(n)=0$
since $\tilde{c}_0\le 1\le l^2$).
\end{proof}

\subsection{Recurrence rate $\rr_\ell(x,\infty,\eps_0=1/2)$ and densities of $\kkk_l$ ($l\ge\ell$)} \label{SUBSECTION:RR-formula}
In this subsection
we give a formula for asymptotic recurrence rate $\rr_\ell(x,\infty,\eps_0)$
based on the densities of the sets $\kkk_l$ ($l\ge \ell$) defined in \eqref{EQ:Kell}. We start with some notation.
For an integer $n\ge 2$ and every $(i,j)\in[0,n)^2$ put
\begin{equation*}
	\ell_{ij}^n = \begin{cases}
		0 &\text{if there is no line in } \recplot(x,n,\eps_0) \text{ starting at } (i,j),
		\\
		\ell &\text{if there is a line in } \recplot(x,n,\eps_0) \text{ starting at } (i,j) \text{ of length }\ell,
	\end{cases}
\end{equation*}
and
\begin{equation*}
	\ell_{ij} = \begin{cases}
		0 &\text{if there is no line in } \recplot(x,\infty,\eps_0) \text{ starting at } (i,j),
		\\
		\ell &\text{if there is a line in } \recplot(x,\infty,\eps_0) \text{ starting at } (i,j) \text{ of length }\ell.
	\end{cases}
\end{equation*}
Clearly, for every $(i,j)\in[0,n)^2$, $\ell_{ij}^n\le \ell_{ij}$ and
\begin{equation}\label{EQ:len_ij}
\begin{split}
	&\ell_{ij}^n = \ell_{ij}
	\quad\text{for every line }(i,j,\ell_{ij}^n) \text{ in } \recplot(x,n,\eps_0)
	\text{ which is not }n\text{-boundary},	
\\
	&\ell_{ij}^n < \ell_{ij}	
	\quad\text{implies that }(i,j,\ell_{ij}^n) \text{ is an }n\text{-boundary line in }\recplot(x,n,\eps_0).
\end{split}
\end{equation}

\begin{proposition}\label{PROP:RR-and-densK}
	For every integer $\ell \geq 1$,  the recurrence rate $\rr_\ell(x,\infty,\eps)$
	exists and
	\begin{equation}\label{EQ:RR-and-densK}
		\rr_\ell(x,\infty,\eps_0)
		=
		\sum_{l=\ell}^{\infty}l \dens(\kkk_{l})
		> 0.
	\end{equation}
\end{proposition}

\begin{proof}
Fix $\ell \geq 1$. The fact that the sum in \eqref{EQ:RR-and-densK} is positive follows from Theorem~\ref{THM:density}; thus it suffices to prove the equality in \eqref{EQ:RR-and-densK}.
Recall from Subsection~\ref{SUBS:RQA} the definitions of $\linedens_{l} (x,n, \eps_0)$ and
$$
	\rr_\ell(x,\infty,\eps_0) 
	= 
	\lim\limits_{n \to \infty} \rr_\ell(x,n,\eps_0) 
	= 
	\lim\limits_{n \to \infty} \sum_{l \geq \ell} l\linedens_{l} (x,n, \eps_0)
$$
(provided the limits exist).
By Lemma~\ref{LMM:moore-osgood} it suffices to show that
\begin{equation}\label{EQ:RR-and-densK:limit0}
	\lim\limits_{n \to \infty} \sum_{l \geq \ell} 
		l \cdot 
		\big|
			\linedens_{l} (x,n, \eps_0)
			- \delta_l(n)
		\big|
	= 0.
\end{equation}

Before diving into the proof, we introduce some notation for every $n\ge 2$ and $l\ge\ell$.
Let $b_l(n)$ denote the number of $n$-boundary lines $(i,j,\ell_{ij}^n)$ in $\recplot(x,n,\eps_0)$ such that $\ell_{ij}=l$; clearly, for $l<n$, 
$$ 
	b_l(n) = \card\Big(
		(\kkk_l\sqcup\kkk_l^0) 
		\cap 
		\big([0, n)^2 \backslash [0, n-l)^2 \big)
	\Big).
$$
Let $S_l(n)$ be the number of recurrences 
in $\recplot(x,n,\eps_0)$
in lines of length $l$,
and 
$T_l(n)$ be the number of recurrences 
in $\recplot(x,\infty,\eps_0)$
in lines of length $l$ starting in $[0,n)^2$; that is,
\begin{eqnarray*}
	S_l(n) 
	&=& 
	l\cdot \numberofdiag_{l}(x, n, \eps_0)
	=
	l(n^2-n) \linedens_l(x,n,\eps_0)
	=
	l\cdot \card\{
		(i,j)\in[0,n)^2\colon
		\ell_{ij}^n=l
	\},
\\
	T_l(n) 
	&=& 
	l(n^2-n)  \big(\delta_l(n) + \delta_l^0(n)\big)
	=
	l\cdot \card\{
		(i,j)\in[0,n)^2\colon
		\ell_{ij}=l
	\},
\end{eqnarray*}
where $\delta_l(n)$ was defined in Lemma~\ref{LMM:ohranicenie-zhora} and 
$\delta_l^0(n)$ is defined analogously by 
$$
	\delta_l^0(n) = \dfrac{\card (\kkk_l^0 \cap [0, n)^2)}{n^{2} - n} \, .
$$
To prove the proposition, we proceed in several steps.

\smallskip
\textit{Step~1.} We first show that, for every $n\ge 2$,
\begin{equation}\label{EQ:RR-and-densK:T-S} 
	\sum_{l\ge\ell}  \abs{T_l(n) - S_l(n)}
	\ \le\ 
	\sum_{l\ge\ell} 2lb_l(n).
\end{equation}
\smallskip

The fact that $\ell_{ij}^n\le \ell_{ij}$ implies
\begin{equation*}
\begin{split}
	T_l(n) - S_l(n)
	&=
	l\cdot \card\{
		(i,j)\in[0,n)^2\colon
		\ell_{ij} = l > \ell_{ij}^n
	\}
\\
	&-
	l\cdot \card\{
		(i,j)\in[0,n)^2\colon
		\ell_{ij} > l = \ell_{ij}^n
	\}.
\end{split}
\end{equation*}
Hence,
\begin{eqnarray*}
	T_l(n) - S_l(n)
	&=&
	\sum_{\ell_{ij} = l > \ell_{ij}^n}  \ell_{ij}
	-
	\sum_{\ell_{ij} > l = \ell_{ij}^n}  \ell_{ij}^n
\\	
	\abs{T_l(n) - S_l(n)}
	&\le&
	\sum_{\ell_{ij} = l > \ell_{ij}^n}  \ell_{ij}
	+
	\sum_{\ell_{ij} > l = \ell_{ij}^n}  \ell_{ij}^n
\\	
	\sum_{l\ge\ell} \abs{T_l(n) - S_l(n)}
	&\le&
	\sum_{\ell_{ij} \ge \ell,\ \ell_{ij} > \ell_{ij}^n}  \ell_{ij}
	+
	\sum_{\ell_{ij} > \ell_{ij}^n \ge \ell} \ell_{ij}^n
\\
	&\le&
	\sum_{\ell_{ij} \ge \ell,\ \ell_{ij} > \ell_{ij}^n}  \ell_{ij}
	+
	\sum_{\ell_{ij} > \ell_{ij}^n \ge \ell} \ell_{ij}
\\	
	&\le&
	2 \sum_{\ell_{ij} \ge \ell,\ \ell_{ij} > \ell_{ij}^n}  \ell_{ij}
\\	
	&=&
	2 \sum_{l\ge\ell} \sum_{\ell_{ij}=l > \ell_{ij}^n}  \ell_{ij}\,,
\end{eqnarray*}
where the summations on the right
are over all pairs $(i,j)\in[0,n)^2$ satisfying the conditions.
Now \eqref{EQ:RR-and-densK:T-S} follows 
easily using the second part of \eqref{EQ:len_ij}.

\smallskip
\textit{Step~2.} By \cite[Proposition~4.3]{polakova2023symbolic},
for every $l\ge 1$ and $n\ge 2$,
\begin{equation}\label{EQ:RR-and-densK:num_nboundary}
	lb_l(n) < 8\recog n.
\end{equation}

\smallskip
\textit{Step~3.} Finally we prove \eqref{EQ:RR-and-densK:limit0}, from which the proposition follows.
\smallskip

The inequalities \eqref{EQ:RR-and-densK:T-S} and \eqref{EQ:RR-and-densK:num_nboundary},
together with \cite[Propositions~4.1 and 4.2]{polakova2023symbolic}, imply
\begin{equation}\label{EQ:RR-and-densK:main}
\begin{split}
		\sum_{l \geq \ell} 
			l \cdot \big|
				\linedens_{l} (x,n, \eps_0)
				- \delta_l(n) - \delta_l^0(n)
			\big|
	&=
	\frac{1}{n^2-n} \sum_{l \geq \ell} \abs{S_l(n)-T_l(n)}
\\
	&<
	\frac{16\recog n}{n^2-n} \cdot \card \{l\geq \ell\colon b_l(n) > 0\}
\\
	&<
	\frac{32(\recog-1)\recog (1+\log_q n)}{n-1}\,
\end{split}
\end{equation}
which converges to $0$ as $n \to \infty$.
Thus, to prove \eqref{EQ:RR-and-densK:limit0}, it remains to show that
\begin{equation}\label{EQ:RR-and-densK:0-boundary}
	\lim\limits_{n\to\infty}
		\sum_{l \geq \ell}  l \delta_l^0(n)
	= 0.
\end{equation}
But this is easy, since the trivial inequality $\card(\kkk_l^0\cap[0,n)^2) \le 2(n-1)$
and \cite[Proposition~4.2]{polakova2023symbolic} imply
\begin{equation*}
	\Big|
		\sum_{l \geq \ell}  l \delta_l^0(n)
	\Big|
	\le
	\frac{2(n-1)}{n^2-n} \card\{l\ge\ell\colon \kkk_l^0\cap[0,n)^2\ne\emptyset\}
	\le\frac{2}{n}(\recog-1)(1+\log_q n).
\end{equation*}
\end{proof}

\subsection{Other recurrence quantifiers for $\eps_0$ and densities of $\kkk_l$} \label{SUBSECTION:LAVG-ENT-formula}
Now we are going to prove formulas analogous to that from Proposition~\ref{PROP:RR-and-densK}
for other recurrence quantifiers.

\begin{proposition}\label{PROP:linedens-and-densK}
	For every integer $\ell \geq 1$,  $\linedens_\ell(x,\infty,\eps_0)$ and
	$\lineDens_\ell(x,\infty,\eps_0)$ 
	exist and 
	\begin{equation}\label{EQ:linedens-and-densK}
		\linedens_\ell(x,\infty,\eps_0)
		=
		\dens(\kkk_{\ell})
	\quad\text{and}\quad
		\lineDens_\ell(x,\infty,\eps_0)
		=
		\sum_{l=\ell}^{\infty}\dens(\kkk_{l})
		> 0;
	\end{equation}
	further, $\dens(\kkk_{\ell})>0$ for infinitely many $\ell$.
\end{proposition}
\begin{proof}
The fact that the sum in \eqref{EQ:linedens-and-densK} is positive and that $\dens(\kkk_\ell)>0$ for infinitely many $\ell$ follows from Theorem~\ref{THM:density}.

Equation \eqref{EQ:RR-and-densK:limit0} from the proof of Proposition~\ref{PROP:RR-and-densK} implies that
\begin{equation*}
	\lim\limits_{n \to \infty} 
		\big|
			\linedens_{\ell} (x,n, \eps_0)
			- \delta_\ell(n)
		\big|
	\le	
	\lim\limits_{n \to \infty} \sum_{l \geq \ell} 
		\big|
			\linedens_{l} (x,n, \eps_0)
			- \delta_l(n)
		\big|
	= 0.
\end{equation*}
Since $\lim_n \delta_\ell(n)=\dens(\kkk_\ell)$, we have
that the limit
$\lim_n \linedens_{\ell} (x,n, \eps_0)$
exists and is equal to $\dens(\kkk_\ell)$; thus we obtained the formula for $\linedens_\ell(x,\infty,\eps_0)$.
Further, $\lim_n \sum_{l\ge\ell} \delta_l(n)=\sum_{l\ge\ell} \dens(\kkk_{l})$ by Lemma~\ref{LMM:moore-osgood-linedens}, hence
the limit
$\lim_n \sum_{l\ge\ell} \linedens_{l} (x,n, \eps_0)$ exists and
and is equal to the previous one;
thus we also proved the formula for $\lineDens_\ell(x,\infty,\eps_0)$.
\end{proof}

Clearly, Propositions~\ref{PROP:RR-and-densK} and \ref{PROP:linedens-and-densK} enable us to calculate also asymptotic determinism $\deter_\ell(x,\infty,\eps_0)$ and average line length
$\lavg_\ell(x,\infty,\eps_0)$ via densities of the sets $\kkk_l$. 
The following is an immediate corollary of Theorem~\ref{THM:rr-and-lavg-via-corsum} and 
Propositions~\ref{PROP:RR-and-densK} and \ref{PROP:linedens-and-densK}.

\begin{proposition}\label{PROP:corsum-and-densK}
	For every integer $\ell \geq 1$,  $\corsum_\ell(x,\infty,\eps_0)$ exists and 
	\begin{equation*}
		\corsum_\ell(x,\infty,\eps_0)
		=
		\sum_{l=\ell}^{\infty} (l-\ell+1)\dens(\kkk_{l})
		> 0.
	\end{equation*}
\end{proposition}

The final result deals with the formula for entropy of line lengths $\ent_\ell(x,\infty,\eps_0)$.

\begin{proposition}\label{PROP:ent-and-densK}
	For every integer $\ell \geq 1$,  the entropy of line lengths $\ent_\ell(x,\infty,\eps_0)$ 
	exists and 
	\begin{equation*}
		\ent_\ell(x,\infty,\eps_0)
		=
		\log\Big(
		  \sum_{l=\ell}^{\infty}\dens(\kkk_{l})
		\Big)
		-
		\frac{1}{\sum_{l=\ell}^{\infty}\dens(\kkk_{l})}
		\sum_{l=\ell}^{\infty}\dens(\kkk_{l})\log \dens(\kkk_{l})
		> 0.
	\end{equation*}
\end{proposition}
\begin{proof}
	Define $\psi(x)=x\log(x)$ for $x\ge 0$ (with the convention $0\log 0$ applied); 
	by Lagrange mean value theorem, 
	\begin{equation}\label{EQ:ent-and-densK:psi}
		\abs{\psi(b) - \psi(a)} 
		\le
		\abs{b-a} \cdot \big(- \log(\min\{a,b\}) + 1  \big)
	\end{equation}
	for every $a,b\in [0,1]$.
	
	By the definition of $\ent_\ell$, Proposition~\ref{PROP:linedens-and-densK} and Lemma~\ref{LMM:moore-osgood-ent}, it suffices to show that
	\begin{equation}\label{EQ:ent-and-densK:main}
		\lim_{n\to\infty} \sum_{l\ge\ell} 
		\big|
			\psi(\linedens_l(x,n,\eps_0)) 
			- 
			\psi(\delta_l(n))
		\big|
		\ = \ 0.
	\end{equation}
	Fix $n\ge 2$ and recall the definitions of $\linedens_l(x,n,\eps_0)$ and $\delta_l(n)$.
	If both these quantities are nonzero, they are
	larger than or equal to $1/(n^2-n)$ and so, by \eqref{EQ:ent-and-densK:psi},
	\begin{equation*}
		\big|
			\psi(\linedens_l(x,n,\eps_0)) 
			- 
			\psi(\delta_l(n))
		\big|		
		\le 
		\big|
			\linedens_l(x,n,\eps_0)
			- 
			\delta_l(n)
		\big|		
		\cdot (2\log n + 1)   \,.
	\end{equation*}
	This inequality remains true also in the case when at least one of $\linedens_l(x,n,\eps_0)$ and $\delta_l(n)$ is zero.
	Hence
	\begin{equation*}
		\sum_{l\ge \ell}\big|
			\psi(\linedens_l(x,n,\eps_0)) 
			- 
			\psi(\delta_l(n))
		\big|		
		\le 
		(2\log n + 1) \sum_{l\ge \ell}\big|
			\linedens_l(x,n,\eps_0)
			- 
			\delta_l(n)
		\big|		  \,.
	\end{equation*}
	To obtain \eqref{EQ:ent-and-densK:main}, it suffices to employ \eqref{EQ:RR-and-densK:main} and \eqref{EQ:RR-and-densK:0-boundary} from Step~3 of the proof of Proposition~\ref{PROP:RR-and-densK} .
	
	It remains to show that $\ent_\ell(\infty,\eps_0)$ is strictly positive. 
	For every $l\ge\ell$ put 
	$z_l=\dens(\kkk_l) / \big(\sum_{l\ge\ell}\dens(\kkk_l)\big)$, and 
	notice that $z_l\ge 0$ and $\sum_{l\ge\ell} z_l=1$.
	By Proposition~\ref{PROP:linedens-and-densK}, $z_l>0$ for infinitely many $l\ge\ell$.
	Thus $0<z_l<1$ for some $l\ge\ell$ and so
	$\sum_{l\ge\ell} \psi(z_l) < 0$.
\end{proof}

\subsection{Recurrence quantifiers for general $\eps$ and embedding dimension $m$: proof of Theorem~\ref{T:formulas-via-infinite-sums}}

\begin{proof}[Proof of Theorem~\ref{T:formulas-via-infinite-sums}]
We may assume that $\zeta$ satisfy also \eqref{assumpt-start-with-0}.
Indeed, if $\zeta(0)$ does not start with $0$, two cases can happen. First, if $\zeta(1)$ starts with $1$,
we can ``switch'' $0$ with $1$; this has no effect on the recurrence plot and recurrence quantifiers.
Second, if $\zeta(1)$ starts with $0$, we can take $\zeta^2$ instead of $\zeta$. Clearly, $\zeta^2$ is primitive and satisfies \eqref{assumpt-start-with-0}; further, by \cite[Proposition~5.4]{queffelec2010substitution},
$X_{\zeta^2}=X_\zeta$ and so $\zeta^2$ is aperiodic.

By Remark~\ref{REM:rr-and-lavg-via-corsum}, it suffices to prove the formulas only for $y=x=\zeta^\infty(0)$. 
We will prove only the formula for recurrence rate; 
the formulas for $\linedens_\ell$, $\lineDens_\ell$ and $\ent_\ell$ can be obtained analogously,
and the formula for $\corsum_\ell$ follows from Theorem~\ref{THM:rr-and-lavg-via-corsum}.

Propositions~\ref{PROP:embeddedRecPlot} and \ref{PROP:recPlot-generalEps} yield
(for abbreviation, put $m'=m+h-2$ and $n'=n+m'$)
\begin{eqnarray*}
	\rr_\ell(x^{(m)},\infty,2^{-h}) 
	&=& 
	\lim_{n\to\infty}  \rr_\ell(x^{(m)},n,2^{-h}) 
	= 
	\lim_{n\to\infty}  \rr_\ell(x,n,2^{-m-h+1}) 
\\
	&=&
	\lim_{n\to\infty} 
	\frac{(n')^2-n'}{n^2-n} \,\cdot\,
	\big[ \rr_{\ell'}(x,n',\eps_0)  -  m' \lineDens_{\ell'}(x,n',\eps_0)   \big]
\\
	&=&
	\rr_{\ell'}(x,\infty,\eps_0)  -  m' \lineDens_{\ell'}(x,\infty,\eps_0) \,.
\end{eqnarray*}
Now it suffices to use Propositions~\ref{PROP:RR-and-densK} and \ref{PROP:linedens-and-densK}.
\end{proof}

\subsection{Explicit formulas for recurrence quantifiers}

Since densities of the sets $\kkk_\ell$ for $\ell\ge\recog$ can be obtained from those of the sets
$\kkk_{\ell_0}$ for $\ell_0<\recog$ (see Theorem~\ref{THM:density}), infinite sums in Theorem~\ref{T:formulas-via-infinite-sums}
can be reduced to closed-form formulas for recurrence quantifiers.
To state them we introduce some notation. 
Let $\recog_0$ be the smallest positive integer such that
\begin{equation*}
	\recog_0 q + (\alpha+\beta) \ge\recog. 
\end{equation*}
For every integer $\ell\in[\recog_0,\recog)$ put
\begin{eqnarray*}
	\nu_\ell^{\numberofdiag} 
	&=&
	\sum_{l\in[\recog_0, \ell)} \dens(\kkk_l),
\\	
	\nu_\ell^{\rr} 
	&=&
	\sum_{l\in[\recog_0, \ell)} l \dens(\kkk_l),
\\	
	\nu_\ell^{\ent} 
	&=&
	- \sum_{l\in[\recog_0, \ell)} \dens(\kkk_l) \log\dens(\kkk_l);	
\end{eqnarray*}
further, for every $\omega\in\{\numberofdiag,\rr,\ent\}$ put
\begin{equation*}
	\tilde\nu_\ell^\omega = \nu_\recog^\omega - \nu_\ell^\omega.
\end{equation*}
For any integer $\ell\ge\recog_0$ let $j=j(\ell)\in\NNN_0$ be the smallest integer such that $(\recog-1)q^j + c(q^j-1)\ge\ell$,
and let $\ell_0=\ell_0(\ell)\in[\recog_0,\recog)$ be the smallest integer such that
\begin{equation}\label{EQ:formulas-def-j-ell0}
	\ell_0 q^j + c(q^j - 1)\ge\ell;
\end{equation}
recall that $c=(\alpha+\beta)/(q-1)\in[0,1]$.

\begin{theorem}\label{T:formulas}
	Let $\zeta$ be a primitive aperiodic binary substitution of constant length.
	Fix any $y\in \Sigma$, $m,\ell\in\NNN$ and $\eps=2^{-h}$ ($h\in\NNN$); put
	$\ell'=\ell+m+h-2$.
	If $\ell'\ge\recog_0$, then
	\begin{eqnarray*}
		\linedens_\ell(y^{(m)},\infty,\eps) 
		&=&
		q^{-2j}\dens(\kkk_{\ell_0}),
	\\ 
		\lineDens_\ell(y^{(m)},\infty,\eps) 
		&=&
		\frac{1}{q^{2j}(q^2-1)} \big[
			\nu_{\ell_0}^{\numberofdiag}   
			+ q \tilde\nu_{\ell_0}^{\numberofdiag}
		\big],
	\\ 
		\rr_\ell(y^{(m)},\infty,\eps) 
		&=&
		\frac{1}{q^{j}(q-1)} \big[
			\nu_{\ell_0}^{\rr}   
			+ q \tilde\nu_{\ell_0}^{\rr}
			\ +\ 
			c(\nu_{\ell_0}^{\numberofdiag}   
			+ q \tilde\nu_{\ell_0}^{\numberofdiag})
		\big]
	\\
		&&-c(m+h-2)\cdot\lineDens_\ell(y^{(m)},\infty,\eps);
	\end{eqnarray*}
	the corresponding formulas for recurrence determinism, average line length and correlation sum readily follows.
	
	If $\ell'<\recog_0$, the formulas for these recurrence quantifiers can be obtained from those for $\ell+m+h-2=\recog_0$
	and from the following recurrent relations (where, on the right hand side, we omit the arguments $y^{(m)}$, $\infty$ and $\eps$):
	\begin{eqnarray*}
		\linedens_\ell(y^{(m)}, \infty,\eps)
		&=&
		\dens(\kkk_{\ell'}),
	\\ 
		\lineDens_\ell(y^{(m)}, \infty,\eps)
		&=&
		\lineDens_{\ell+1} + \dens(\kkk_{\ell'})
		,
	\\ 
		\rr_{\ell}(y^{(m)}, \infty,\eps)
		&=&
		\rr_{\ell+1} + \ell \dens(\kkk_{\ell'}).
	\end{eqnarray*}
\end{theorem}

\begin{proof}
It suffices to prove the result in the case $m=h=1$ and $\ell\ge\recog_0$. Put
\begin{equation*}
	\varphi_\ell^{\numberofdiag}=\dens(\kkk_\ell),
	\qquad
	\varphi_\ell^{\rr}=\ell \dens(\kkk_\ell),
\end{equation*}
and, for any $\omega\in\{\numberofdiag,\rr\}$ and $j,\ell_0\in\NNN_0$,
\begin{equation*}
	\Phi_{j,\ell_0}^\omega = \sum_{k\ge j} \varphi_{\ell_0 q^k + c(q^k-1)}^\omega.
\end{equation*} 
By Theorems~\ref{T:formulas-via-infinite-sums} and \ref{THM:density}, 
the quantities $\lineDens_\ell$ and $\rr_\ell$
can be written as (where $\omega$ is $\numberofdiag$ and $\rr$, respectively)
\begin{equation*}
	\sum_{l\in[\recog_0,\ell_0)} \Phi_{j+1,l}^\omega 
	+ \sum_{l\in[\ell_0,\recog)} \Phi_{j,l}^\omega.
\end{equation*}
Now, to obtain the theorem, it suffices to use the fact that 
$\dens(\kkk_{\ell_0 q^k + c(q^k-1)}) = q^{-2k}\dens(\kkk_{\ell_0})$ by Theorem~\ref{THM:density},
and the formulas for sums of geometric and arithmetico-geometric sequences:
\begin{equation*}
	\sum_{k\ge j} x^k = \frac{x^j}{1-x}
	\quad\text{and}\quad
	\sum_{k\ge j} kx^k = \frac{jx^j - (j-1)x^{j+1}}{(1-x)^2}\,.
\end{equation*}
\end{proof}

\begin{theorem}\label{T:formulas-ent}
	Let $\zeta$ be a primitive aperiodic binary substitution of constant length
	satisfying \eqref{assumpt-start-with-0}, and $x=\zeta^\infty(0)$ be a unique fixed point
	of $\zeta$ starting with $0$.
	Fix any $m,\ell\in\NNN$ and $\eps=2^{-h}$ ($h\in\NNN$); put $\ell'=\ell+m+h-2$.
	If $\ell'\ge\recog_0$, then
	\begin{eqnarray*}
		\ent_\ell(x^{(m)},\infty,\eps) 
		&=&
		\log \lineDens_\ell
		-
		\frac{1}{\lineDens_\ell}
		\widetilde{\ent}_\ell(x^{(m)},\infty,\eps) , \qquad\text{where}
	\\ 
		\widetilde{\ent}_\ell(x^{(m)},\infty,\eps) 
		&=&
		\frac{2\log q}{q^{2j}(q^2-1)^2} \big[
			((j+1)q^2-j)\nu_{\ell_0}^{\numberofdiag}   
			+ q^2(jq^2-j+1) \tilde\nu_{\ell_0}^{\numberofdiag}
		\big]	
	\\			
		&&+
		\frac{1}{q^{2j}(q^2-1)} \big[
			\nu_{\ell_0}^{\ent}   
			+ q^2 \tilde\nu_{\ell_0}^{\ent}
		\big].
	\end{eqnarray*}
	If $\ell'<\recog_0$, the formulas for these quantifiers can be obtained from those for $\ell+m+h-2=\recog_0$
	and from the following recurrent relation (where we omit the arguments $x^{(m)}$, $\infty$ and $\eps$):
	\begin{eqnarray*}
		\widetilde{\ent}_\ell
		&=&
		\widetilde{\ent}_{\ell+1} - \dens(\kkk_{\ell'}) \log \dens(\kkk_{\ell'}).
	\end{eqnarray*}
\end{theorem}
\begin{proof}
	The proof is analogous to that of Theorem~\ref{T:formulas},
	with a function $\varphi_\ell^{\ent}=-\dens(\kkk_\ell)\log\dens(\kkk_\ell)$ used
	instead of $\varphi_\ell^\numberofdiag$ and $\varphi_\ell^\rr$.
\end{proof}

Theorems~\ref{T:formulas} and \ref{T:formulas-ent} 
together with the algorithm for calculating the densities of the sets $\kkk_\ell$ \cite[Appendix~B]{polakova2023symbolic}
allow us to obtain explicit formulas for recurrence characteristics for any substitution satisfying the assumptions.
See the next section for some examples.

\section{Examples}\label{S:examples}

\subsection{Explicit RQA formulas for the Thue-Morse substitution}
Let $\zeta$ be the Thue-Morse substitution $\zeta(0)=01$ and $\zeta(1)=10$; then $\recog=4$
since every allowed $4$-word is recognizable and $\alpha=\beta=0$.
The densities of nonempty sets $\kkk_\ell$ are \cite[Example~5.1]{polakova2023symbolic}
\begin{equation*}
	\dens(\kkk_1)=\frac19,
	\quad
	\dens(\kkk_{2^k})=\frac1{9\cdot 2^{2k-1}},
	\quad
	\dens(\kkk_{3\cdot 2^{k-1}})=\frac1{9\cdot 2^{2k}}
\end{equation*}
for every $k\in\NNN$.
Fix any $m,\ell\in\NNN$ and $\eps=2^{-h}$ ($h\in\NNN$). Put $\ell' = \ell + m+h$
and let $j=j(\ell')\ge 0$ and $\ell_0=\ell_0(\ell')\in\{2,3\}$ be given by \eqref{EQ:formulas-def-j-ell0}
if $\ell'\ge\recog_0=2$, and $j=0$, $\ell_0=1$ if $\ell'=1$.
Now Theorems~\ref{T:formulas} and \ref{T:formulas-ent} yield that
\begin{eqnarray*}
	\lineDens_\ell(y^{(m)},\infty,\eps) 
	&=& 
	\frac{a_{\ell_0}}{9 \cdot 2^{2j+1}} \,,
\\	
	\rr_\ell(y^{(m)},\infty,\eps) 
	&=& 
	\frac{b_{\ell_0}}{9 \cdot 2^{j+1}} 
	-  \frac{(m+h-2)a_{\ell_0}}{9 \cdot 2^{2j+1}}
	\,,
\\
	\ent_\ell(x^{(m)},\infty,\eps) 	
	&=&	
	2\log 2,
\end{eqnarray*}
where
\begin{equation*}
	a_{\ell_0} = \begin{cases}
		2 &\text{if } \ell_0=1,
		\\
		2 &\text{if } \ell_0=2,
		\\
		1 &\text{if } \ell_0=3,
	\end{cases}
	\qquad
	b_{\ell_0} = \begin{cases}
		9 &\text{if } \ell_0=1,
		\\
		7 &\text{if } \ell_0=2,
		\\
		5 &\text{if } \ell_0=3.
	\end{cases}
\end{equation*}
Clearly, formulas for $\deter_\ell$, $\lavg_\ell$ and $\corsum_{\ell}$ 
can be easily derived.

\subsection{Explicit RQA formulas for the period-doubling substitution}
Let $\zeta$ be the period-doubling substitution $\zeta(0)=01$ and $\zeta(1)=00$; then $\recog=3$
since every allowed $3$-word is recognizable and $\alpha=1$, $\beta=0$.
The densities of nonempty sets $\kkk_\ell$ are \cite[Theorem~1]{vspitalsky2018recurrence}
\begin{equation*}
	\dens(\kkk_{2^{k+1}-1})=\frac1{9\cdot 2^{2k}},
	\quad
	\dens(\kkk_{3\cdot 2^{k}-1})=\frac1{9\cdot 2^{2k+1}}
\end{equation*}
for every $k\in\NNN_0$.
Fix any $m,\ell\in\NNN$ and $\eps=2^{-h}$ ($h\in\NNN$). Put $\ell' = \ell + m+h$
and let $j=j(\ell')\ge 0$ and $\ell_0=\ell_0(\ell')\in\{1,2\}$ be given by \eqref{EQ:formulas-def-j-ell0}; note that $\recog_0=1$.
Now Theorems~\ref{T:formulas} and \ref{T:formulas-ent} yield that
(compare with \cite{vspitalsky2018recurrence, polakova2020formulas})
\begin{eqnarray*} 
	\lineDens_\ell(y^{(m)},\infty,\eps) 
	&=& \frac{a_{\ell_0}}{9 \cdot 2^{2j}} \,,
\\	
	\rr_\ell(y^{(m)},\infty,\eps) 
	&=& 
	\frac{b_{\ell_0}}{9 \cdot 2^j} 
	- \frac{(m+h-1)a_{\ell_0}}{9 \cdot 2^{2j}}
	\,,
\\
	\ent_\ell(x^{(m)},\infty,\eps) 	
	&=&	
	2\log 2,
\end{eqnarray*}
where
\begin{equation*}
	a_{\ell_0} = \begin{cases}
		2 &\text{if } \ell_0=1,
		\\
		1 &\text{if } \ell_0=2,
	\end{cases}
	\qquad
	b_{\ell_0} = \begin{cases}
		7 &\text{if } \ell_0=1,
		\\
		5 &\text{if } \ell_0=2.
	\end{cases}
\end{equation*}

\subsection{Explicit RQA formulas for a substitution of length $5$}
Let $\zeta$ be the substitution $\zeta(0)=01110$ and $\zeta(1)=01010$; then $\recog=5$
since every allowed $4$-word is recognizable and $\alpha=\beta=2$.
The densities of nonempty sets $\kkk_\ell$ are \cite[Example~5.2]{polakova2023symbolic}
\begin{equation*}
\begin{split}
	&\dens(\kkk_{2\cdot 5^k-1})=\frac7{2\cdot 5^{2k+2}}\,,
	\quad
	\dens(\kkk_{3\cdot 5^k-1})=\frac3{2\cdot 5^{2k+2}}\,,
\\	
	&\dens(\kkk_{4\cdot 5^k-1})=\frac1{2\cdot 5^{2k+2}}\,,
	\quad
	\dens(\kkk_{5\cdot 5^k-1})=\frac{13}{2\cdot 5^{2k+4}}
\end{split}
\end{equation*}
for every $k\in\NNN_0$. 

Fix any $m,\ell\in\NNN$ and $\eps=2^{-h}$ ($h\in\NNN$). Put $\ell' = \ell + m+h$
and let $j=j(\ell')\ge 0$ and $\ell_0=\ell_0(\ell')\in\{1,2\}$ be given by \eqref{EQ:formulas-def-j-ell0}; note that $\recog_0=1$.
Theorem~\ref{T:formulas} yields that, for every $y\in X_\zeta$, 
\begin{eqnarray*}
	\lineDens_\ell(y^{(m)},\infty,\eps) 
	&=& \frac{a_{\ell_0}}{6 \cdot 5^{2j+2}} \,,
\\	
	\rr_\ell(y^{(m)},\infty,\eps) 
	&=& 
	\frac{b_{\ell_0}}{2 \cdot 5^{j+2}} 
	- 
	\frac{(m+h-1)a_{\ell_0}}{6 \cdot 5^{2j+2}}
	\,,
\end{eqnarray*}
where
\begin{equation*}
	a_{\ell_0} = \begin{cases}
		36 &\text{if } \ell_0=1,
		\\
		15 &\text{if } \ell_0=2,
		\\
		6 &\text{if } \ell_0=3,
		\\
		4 &\text{if } \ell_0=4,
	\end{cases}
	\qquad
	b_{\ell_0} = \begin{cases}
		37 &\text{if } \ell_0=1,
		\\
		23 &\text{if } \ell_0=2,
		\\
		14 &\text{if } \ell_0=3,
		\\
		10 &\text{if } \ell_0=4.
	\end{cases}
\end{equation*}
A (lengthy) formula for entropy of line lengths is omitted.

\section{Recurrence quantification analysis for non-primitive substitutions; proof of Theorem~\ref{THM:nonPrimitive}}\label{S:nonPrimit}

In this section we describe recurrence characteristics of substitution shifts given by non-primitive binary substitutions.
Recall that a dynamical system $(X,f)$ is \emph{proximal} if every pair $(x,y)\in X\times X$ is proximal, 
that is, 
\begin{equation*}
	\liminf_{n\to\infty} \rho(f^n(x), f^n(y)) = 0.
\end{equation*} 
By \cite[Proposition~2.2]{akin2003li}, a system $(X,f)$ is proximal if and only if it has a fixed point which is the unique minimal set of $(X,f)$.

\begin{proposition}\label{P:nonprimit-uniq-ergod}
	Let $\zeta$ be a binary substitution of constant length $q\ge 2$ such that
	\begin{equation}\label{LAB:P:nonprimit-uniq-ergod:zeta} 
		\text{either }\zeta(0)\ne 0^q \text{ and }\zeta(1)=1^q,
		\quad\text{or }\zeta(0)= 0^q\text{ and }\zeta(1)\ne 1^q.
	\end{equation}
	Then the subshift $(X_\zeta,\sigma)$ is uniquely ergodic and 
	its unique invariant measure is the Dirac measure $\delta_{1^\infty}$.	
	Consequently, $(X_\zeta,\sigma)$ is proximal and $\{1^\infty\}$ is a unique minimal set of it.
\end{proposition}
\begin{proof}
We may assume that $\zeta(0)\ne 0^q$ and $\zeta(1)=1^q$.
For any $w\in A^*$, let $N(w)$ denote the number of zeros in $w$. Put $p=N(\zeta(0))$; then $p\le q-1$. To prove the proposition, we proceed in four steps.

\smallskip
\textit{Step~1.} $N(\zeta^k(0))=p^k$ for every $k\in\NNN$.
\smallskip

This follows from a simple observation that, due to the fact that $\zeta(1)$ does not contain $0$, 
$N(\zeta(w))=pN(w)$ for every $w\in A^*$.

\smallskip
\textit{Step~2.} For $\ell\in\NNN$ put $\triangle_\ell=\max\{N(w)\colon w\in\lang_\zeta,\ \abs{w}=\ell\}$. We claim that
\begin{equation*}
	\lim_{\ell\to\infty} \frac{\triangle_\ell}{\ell} = 0.
\end{equation*}
\smallskip

To prove this, fix any $\ell\ge 2q-1$. Let $k=k_\ell$ be the smallest positive integer such that $\ell\ge 2q^k-1$;
that is, $k=\lfloor \log_q((\ell+1)/2) \rfloor$. Fix any 
$\ell$-word $w\in\lang_\zeta$. Then, by the definition of $X_\zeta$, 
there are $a\in A$, $m\ge 0$ and $i\in[0,q^m-\ell]$ such that $w=\zeta^m(a)_{[i,i+\ell)}$.
Clearly, $m>k$ since $q^m\ge\ell$; put $u=\zeta^{m-k}(a)$.
Let $j,h\in\NNN$ be unique integers such that
\begin{equation*}
	(j-1)q^k\le i < jq^k
	\quad\text{and}\quad
	(j+h)q^k <  i+\ell \le (j+h+1) q^k.
\end{equation*}
Then $h\le 2q-1$ and $w$ is a subword of $\zeta^k(u_{[j-1,j+h+1)})$. So, by Step~1, $N(w)\le (h+2)p^k\le (2q+1)(q-1)^k$.
Since this is true for every $\ell$-word $w\in\lang_\zeta$, we have
$\triangle_\ell/\ell \le  (2q+1)(q-1)^{k_\ell} / (2q^{k_\ell}-1) \to 0$ as $\ell\to\infty$.

\smallskip
\textit{Step~3.} We prove that $\delta_{1^\infty}$ is a unique invariant measure of $(X_\zeta,\sigma)$.
\smallskip

Let $\mu$ be any invariant measure of $(X_\zeta,\sigma)$. Since, by Step~2, the upper Banach density of
the set $\{i\in\NNN\colon x_i=0\}$ is zero, $\mu([0])=0$ by \cite[Lemma~3.17]{furstenberg1981recurrence}.
Since $\mu$ is $\sigma$-invariant, $\mu([w])=0$ for every $w\in \lang_\zeta$ containing zero. 
(In fact, this follows by induction using that 
$\mu([1w])=\mu([1w])+\mu([0w])=\mu(\sigma^{-1}[w])=\mu([w])$ for every $w$.) So $\mu=\delta_{1^\infty}$.

\smallskip
\textit{Step~4.} We finish the proof.
\smallskip

Let $M$ be a minimal set of $(X_\zeta,\sigma)$. Fix an invariant measure $\mu$ of $(M,\sigma|_M)$;
then the support $\supp(\mu)$ of $\mu$ is $M$. Since $\mu$ is also an invariant measure of $(X_\zeta,\sigma)$,
we have that $\mu=\delta_{1^\infty}$ and so $M=\supp(\mu)=\{1^\infty\}$. By \cite[Proposition~2.2]{akin2003li},
$(X_\zeta,\sigma)$ is proximal.
\end{proof}

\begin{remark}\label{REM:nonPrimitive}
If $\zeta$ is not primitive then at least one of $\zeta(0)$, $\zeta(1)$ belongs to 
$\{0^q,1^q\}$; moreover, if $\zeta(0)=1^q$ or $\zeta(1)=0^q$, then 
both $\zeta(0)$ and $\zeta(1)$ belongs to $\{0^q,1^q\}$.
Thus, if non-primitive $\zeta$ does not satisfy 
\eqref{LAB:P:nonprimit-uniq-ergod:zeta},
then either $\zeta(0)=0^q$ and $\zeta(1)=1^q$, or  $\zeta(0)=1^q$ and $\zeta(1)=0^q$.
In both cases $X_\zeta=\{0^\infty,1^\infty\}$ and the $\zeta$-subshift is not uniquely ergodic.
\end{remark}

\begin{proof}[Proof of Theorem~\ref{THM:nonPrimitive}]
	If $\zeta$ does not satisfy 
	\eqref{LAB:P:nonprimit-uniq-ergod:zeta},
	then $X_\zeta=\{0^\infty,1^\infty\}$ by Remark~\ref{REM:nonPrimitive}
	and the statement of the theorem is trivial. So we may assume that $\zeta$ satisfies 
		\eqref{LAB:P:nonprimit-uniq-ergod:zeta}.
	
	Put $\mu=\delta_{1^\infty}$ and fix any $y\in X_\zeta$,  $m,\ell\in\NNN$ and $\eps>0$; 
	we may assume that $\eps<1$. Since $(X_\zeta,\sigma)$ is uniquely ergodic
	by Proposition~\ref{P:nonprimit-uniq-ergod}, 
	\begin{equation*}
		\corsum_\ell(y^{(m)},\infty,\eps)
		= \mu\times\mu\{
			(z,z')\in\Sigma\times\Sigma\colon \rho_\ell(z,z')\le 2^{-m+1}\eps
		\}
		=1
	\end{equation*}
	by Proposition~\ref{PROP:embeddedCorSum} and \eqref{EQ:corr-integral}.
	This and Theorem~\ref{THM:rr-and-lavg-via-corsum} 
	imply that $\rr_\ell(y^{(m)},\infty,\eps)=1$. Since also $\rr_1(y^{(m)},\infty,\eps)=1$, 
	we have $\deter_\ell(y^{(m)},\infty,\eps)=1$.
	Further, $\lavg_\ell(y^{(m)},\infty,\eps)=\infty$ by Theorem~\ref{THM:rr-and-lavg-via-corsum}.
\end{proof}

We do not know whether, under the assumptions of Theorem~\ref{THM:nonPrimitive}, the entropy of line lengths $\ent_\ell(y^{(m)},\infty,\eps)$ is infinite.

\section{Proof of Theorem~\ref{T:determinism}}\label{S:limOfDet}

In this section, we give a proof of Theorem~\ref{T:determinism} stating that the determinism of any uniform binary substitution converges to $1$ as $\eps\to 0$. 
We start with a lemma for primitive aperiodic substitutions.

\begin{lemma}\label{LMM:RR/RR-to-infty}
	Let $\zeta$ be a primitive aperiodic binary substitution of constant length.
	Then, for every $y\in X_\zeta$ and $m,\ell \in\NNN$,
	$$
		\lim_{h \to \infty}\dfrac{\rr_{\ell+1} (y^{(m)},\infty, 2^{-h})}{\rr_\ell(y^{(m)},\infty, 2^{-h})} = 1.
	$$
\end{lemma}

\begin{proof}
	Put $\eps = 2^{-h}$ and $k=m+h-2$;
	we may assume that $h$ is such that $\ell+k \geq \rho$.
	By Theorem~\ref{T:formulas-via-infinite-sums}, 
	\begin{equation*}
		\dfrac
			{\rr_{\ell+1} (y^{(m)},\infty, \eps)}
			{\rr_\ell(y^{(m)},\infty, \eps)}
		=
		1 - 
		\dfrac
			{\ell \dens(\kkk_{\ell+k})}
			{\sum\limits_{l=\ell}^{\infty} l \dens(\kkk_{l+k})} 
		\,.
	\end{equation*}
	If $\dens(\kkk_{\ell+k}) = 0$, then $\rr_{\ell+1}/ \rr_\ell = 1$. 
	Assume that	$\dens(\kkk_{\ell+k}) \neq 0$.
	Put $l_0 = \ell+k$ and $l_1 = q l_0 + \alpha + \beta$.
	By Theorem~\ref{THM:density}, $\dens(\kkk_{l_1}) = \dens(\kkk_{l_0})/q^2$.
	Thus
	\begin{equation*}
		\dfrac
			{\ell \dens(\kkk_{\ell+k})}
			{\sum\limits_{l=\ell}^{\infty} l \dens(\kkk_{l+k})} 
		\leq
		\dfrac
			{\ell \dens(\kkk_{l_0})}
			{(l_1-l_0+\ell)\dens(\kkk_{l_1})}
		= 
		\dfrac
			{\ell q^2}
			{(q-1) l_0 + \alpha + \beta +\ell} 
		<
		\dfrac{\ell q^2}{h}.
	\end{equation*}
	From this estimate the result immediately follows.
\end{proof}

\begin{proof}[Proof of Theorem \ref{T:determinism}]
If $\zeta$ is not primitive, \eqref{EQ:determinism} is trivial by Theorem~\ref{THM:nonPrimitive}.
If $\zeta$ is primitive but not aperiodic, then every $y\in X_\zeta$ is $\sigma$-periodic;
thus $\deter_\ell(y^{(m)},\infty,\eps)=1$ for every sufficiently small $\eps>0$ 
and \eqref{EQ:determinism} is again trivial.
Finally, if $\zeta$ is primitive and aperiodic, it suffices to use
\begin{equation*}
	\deter_\ell = \prod_{k=1}^{\ell-1} \dfrac{\rr_{k+1}}{\rr_k}
\end{equation*}
and Lemma~\ref{LMM:RR/RR-to-infty}.
\end{proof}


\section*{Acknowledgments}
This work was supported by VEGA grant 1/0158/20.

\bibliography{subst-rqa}

\end{document}